\documentclass{amsproc}
\usepackage{euscript}
\usepackage{cases}
\usepackage{mathrsfs}
\usepackage{bbm}
\usepackage{amssymb}
\usepackage{amsfonts,amsmath,amsxtra,mathdots,mathabx,mathtools}
\usepackage[mathlines]{lineno}
\usepackage{color}
\usepackage{hyperref}
\usepackage{tikz}
\usepackage{appendix,upgreek}

\allowdisplaybreaks

\DeclareFontFamily{U}{matha}{\hyphenchar\font45}
\DeclareFontShape{U}{matha}{m}{n}{
	<5> <6> <7> <8> <9> <10> gen * matha
	<10.95> matha10 <12> <14.4> <17.28> <20.74> <24.88> matha12
}{}
\DeclareSymbolFont{matha}{U}{matha}{m}{n}

\DeclareMathSymbol{\Lt}{3}{matha}{"CE}
\DeclareMathSymbol{\Gt}{3}{matha}{"CF}

\DeclareSymbolFont{mathc}{OML}{txmi}{m}{it}
\DeclareMathSymbol{\varuu}{\mathord}{mathc}{117}
\DeclareMathSymbol{\varvv}{\mathord}{mathc}{118}
\DeclareMathSymbol{\varww}{\mathord}{mathc}{119}

\def\SB{\text{\raisebox{- 2 \depth}{\scalebox{1.1}{$ \text{\usefont{U}{BOONDOX-calo}{m}{n}B}   $}}}}

\def\SX{\text{\raisebox{- 2 \depth}{\scalebox{1.1}{$ \text{\usefont{U}{BOONDOX-calo}{m}{n}X}   $}}}}

\def\valpha{\text{\scalebox{0.86}[1.02]{$\alpha$}}}   
\def\vepsilon{\upvarepsilon}
\def\vnu{\text{{\scalebox{0.86}[1]{$\nu$}}}} 
\def\vkappa{\text{{\scalebox{0.86}[1.1]{$\kappa$}}}} 
\def\vchi{\text{\scalebox{0.9}[1.06]{$\chi$}}}

\newcommand{\BC}{{\mathbf {C}}}

\newcommand{\BR}{{\mathbf {R}}} 
\newcommand{\BZ}{{\mathbf {Z}}}

\newcommand{\ra}{\rightarrow} 
\def\sumx{\sideset{}{^\star}\sum}
\def\sumd{\sideset{}{^{\delta}}\sum}

\def\nd{\mathrm{d}}

\def\lp {\left (}
\def\rp {\right )}

\def\shskip{\hspace{0.5pt}}

\newcommand{\delete}[1]{}

\theoremstyle{plain}

\newtheorem{coro}{Corollary}[section]
\newtheorem{lem}{Lemma}[section]
\newtheorem{theorem}{Theorem}[section] 
\newtheorem{proposition}{Proposition}[section]

\newtheorem*{thm*}{Theorem}

\theoremstyle{remark} 
\newtheorem{remark}{Remark}[section] 
\newtheorem{defn}{Definition}[section]

\numberwithin{equation}{section}

 \usepackage[mathlines]{lineno}

\begin{document}
	

	\title[Secondary Term for the Mean Value of Maass Special $L$-values]{{Secondary Term for the Mean Value of Maass Special $L$-values}}

	
	\begin{abstract}
	 In this paper, we discover a secondary term in the asymptotic formula for the mean value of Hecke--Maass special $L$-values $ L (1/2+it_f, f) $ with the average over $f (z)$ in an orthonormal basis of Hecke--Maass cusp forms of Laplace eigenvalue $1/4 + t_f^2$ ($t_f > 0$).  To be explicit, we prove
	 $$ \sum_{t_f \leqslant T} \omega_f L (1/2+it_f, f) = \frac {T^2} {\pi^2}   + \frac {4 T^{3/2}} {3\pi^{3/2} }  + O \big(T^{1+\upvarepsilon}\big),  $$
	for any $\upvarepsilon > 0$, where $\omega_f$ are the harmonic weights. This provides a new instance of (large) secondary terms in the moments of $L$-functions---it was known previously only for the smoothed cubic moment of quadratic Dirichlet $L$-functions. The proof relies on an explicit formula for the smoothed mean value of $L (1/2+it_f, f)$. 
	
	\end{abstract}
	
	\author{Zhi Qi}
	\address{School of Mathematical Sciences\\ Zhejiang University\\Hangzhou, 310027\\China}
	\email{zhi.qi@zju.edu.cn}
	
	\thanks{The author was supported by National Key R\&D Program of China No. 2022YFA1005300.}

	
	\maketitle

	{\small \tableofcontents}

	\section{Introduction}\label{sec: intro}
	
	

Let $ \SB $ be an orthonormal basis of Hecke--Maass  cusp forms  on the modular surface $\mathrm{SL}_2  ( \BZ) \backslash \mathbf{H} $.  Assume that each $f \in \SB $ is either even or odd  in the sense that $ f (- \widebar{z}) = (-1)^{\delta_f} f  (z)$ for $\delta_f = 0$ or $ 1$.  Let  $ \SB_0 $ or $\SB_1 $ be the subset of even or odd forms in $ \SB $ respectively. For $f \in \SB $, let   $\lambda_{f} = s_{f} (1-s_{f})$   
be its Laplace eigenvalue, 
with $s_f = 1/2+ i t_f$ ($t_f > 0$). 
The Fourier expansion of $f (z)$ reads:
\begin{align*}
	f (x+iy) =   \sqrt{y} \sum_{n \neq 0}  \rho_f (n) K_{i t_f} (2\pi |n| y) e (n x), 
\end{align*}
where as usual $K_{\vnu} (x)$ is the $K$-Bessel function and $e (x) = \exp (2\pi i x)$.  Let $\lambda_f (n)$ ($n \geqslant 1$) be the Hecke eigenvalues of $f (z)$. It is well known that  $\rho_f (\pm n) =   \allowbreak \rho_f (\pm 1)  \lambda_f (n) $, while $\rho_f (-1) = (-1)^{\delta_f} \rho_f (1)$.  
Define the harmonic weight 
\begin{align*}
	\omega_f = \frac {|\rho_f (1)|^2} {\cosh \pi t_f} .    
\end{align*}
The Hecke--Maass $L$-function of $f (z)$ is defined by
\begin{align*}
L (s, f) = \sum_{n    =1}^{\infty}\frac{ \lambda_f (n)  }{n^s} ,
\end{align*}
for $\mathrm{Re} (s) > 1$, and by analytic continuation on the whole complex plane.

Recently, the author proved an asymptotic formula for the mean value of $L (s_f, f)$ (see \cite[Corollary 8]{Qi-GL(2)-Special}): 
\begin{align*}
	\sumd_{ t_f \leqslant T } \omega_f L (s_f, f) =   \frac {T^{2}} {2 \pi^2}   + O_{  \vepsilon} \big(T^{3/2+\vepsilon}\big), 
\end{align*}
where the superscript $\delta$ restricts the sum on $  \SB_{\delta}$. Consequently,
\begin{align*}
	\sum_{ t_f \leqslant T } \omega_f L (s_f, f) =   \frac {T^{2}} {\pi^2}  + O_{  \vepsilon} \big(T^{3/2+\vepsilon}\big), 
\end{align*}
 if the sum is on the entire $ \SB $. 
 
 The main purpose of this paper is to refine  these asymptotic formulae:  extract a certain secondary term of order $T^{3/2}$ in the even case of $\SB_0$ and improve the error term into $O (T^{1+\vepsilon})$ in the odd case of $\SB_1$. 
 
 \begin{theorem}\label{thm: asymp} 
 	We have 
 	\begin{align}
 		\label{1eq: main asymp, 0}
 		\sideset{}{^{\delta}}\sum_{ t_f \leqslant T } \omega_f L (s_f, f) =   \frac {T^{2}} {2 \pi^2}   + (1-\delta)  \frac {4 T^{3/2}} {3 \pi^{3/2} }    +  O_{  \vepsilon}   \big(T^{1+\vepsilon} \big),  
 	\end{align} 
 and hence 
 \begin{align}
 	\label{1eq: main asymp}
 	\sum_{ t_f \leqslant T } \omega_f L (s_f, f) =   \frac {T^{2}} {  \pi^2}   + \frac {4 T^{3/2}} {3 \pi^{3/2} }  +  O_{  \vepsilon}   \big(T^{1+\vepsilon} \big) .
 \end{align} 
 \end{theorem}


\subsection{Previous Example: Secondary Terms in the Moments of Quadratic Dirichlet $L$-functions}  

The moments of quadratic Dirichlet $L$-functions have been  investigated by many mathematicians since the 1980s. The reader is referred to  \cite{Jutila-Mean-Value-Quad,VT-Quad-Dirichlet,Goldfeld-Hoffstein-Dirichlet,Sound-Non-vanishing,D-Goldfeld-H-Multiple-Dirichlet,Zhang-Dirichlet-3/4,Zhang-Dirichlet-Survey,Young-First-Moment-Dirichlet,Young-Cubic-Moment-Dirichlet,Sono-2nd,Diaconu-Whitehead-3rd,Shen-4th,Diaconu-Twiss-2nd-Terms} and for the function field analogues to \cite{Andrade-Keating-1st-FF,Florea-1st-FF,Florea-2nd-3rd-FF,Florea-4th-FF,Diaconu-3rd-FF}.

In the work of Conrey, Farmer, Keating, Rubinstein, and Snaith \cite{Conrey-FKRS-Moments}, the Random Matrix Theory (RMT) was used to predict all of the main terms in the moments of various families of primitive $L$-functions. As an important (symplectic) example,  the authors conjectured the asymptotic formula \cite[(1.3.4)]{Conrey-FKRS-Moments}:
\begin{align}\label{1eq: conj, quad}
	\sumx_{ |d| \leqslant X} L (1/2, \vchi_d)^{k} = \frac {6} {\pi^2} X \mathcal{Q}_k (\log X) + O_{\vepsilon} \big(X^{1/2+\vepsilon} \big), 
\end{align}
where the $\star$ sum  is over fundamental discriminants, the real character $ \vchi_d (n) = (d/n) $ is the Kronecker symbol, and $\mathcal{Q}_k$ is a polynomial of degree $k(k+1)/2$. 

For simplicity, many authors considered the smoothed moments restricted to the subset of positive discriminants divisible by $8$. 

For $k = 1$ and $2$, the error bound $O  (X^{1/2+\vepsilon}  )$ as in \eqref{1eq: conj, quad} has been proven for its smoothed variant by Young \cite{Young-First-Moment-Dirichlet} (implicitly in \cite{Goldfeld-Hoffstein-Dirichlet}) and Sono \cite{Sono-2nd}. 

For $k =3$, however, Diaconu, Goldfeld, and Hoffstein \cite{D-Goldfeld-H-Multiple-Dirichlet} used the method of Multiple Dirichlet Series (MDS) and conjectured the existence of a secondary term of order $ X^{3/4}$, which was subsequently verified conditionally by Zhang \cite{Zhang-Dirichlet-3/4} and calculated numerically by Alderson and Rubinstein \cite{A-R-Dirichlet-Experiments}, while the error bound $O (X^{3/4+\vepsilon})$ was proven in the smoothed case by Young \cite{Young-Cubic-Moment-Dirichlet}.  Recently, Diaconu and Whitehead \cite{Diaconu-Whitehead-3rd} finally proved that the secondary term indeed exists: 
\begin{align}\label{1eq: quad, 3rd moment}
	\sumx_{8 | d} L (1/2, \vchi_{d} )^3 W (d/X) = X \mathcal{Q}_{W} (\log X) + \mathcal{C}_W X^{3/4} + O_{\vepsilon, W} \big(X^{2/3+\vepsilon}\big), 
\end{align}
where  the degree 6 polynomial $\mathcal{Q}_{W}$ and the constant $\mathcal{C}_W$ are computable in terms of the compactly supported smooth weight $W: (0, \infty) \ra [0, 1]$.

For $k \geqslant 4$, it is conjectured that there exist infinitely many main terms of order $X^{ 1 / 2+   1 / {2 l}}$ ($l = 1, 2$, ...) 
as described in \cite[Conjecture 3.3]{Zhang-Dirichlet-Survey} (see also \cite{D-Goldfeld-H-Multiple-Dirichlet,Diaconu-Twiss-2nd-Terms}). 

Nevertheless, to date  no other example  has exhibited such a large secondary term as in \eqref{1eq: quad, 3rd moment}, and, as far as the author is aware, it is not yet clear how to incorporate this into the framework of moment conjectures.

\subsection{Remarks} 
Although the $L$-values are of quite different nature, our \eqref{1eq: main asymp} is very similar to \eqref{1eq: quad, 3rd moment} and provides a new example of secondary term of the same type. However, unlike the case of $L (1/2, \vchi_d )$, it is curious that such a term arises in the {\it first} moment of $L(s_f, f)$. 
Hopefully, the study of (higher) moments of $L(s_f, f)$  will shed light on the moment conjectures from a more analytic perspective. 

So far,  the mean-Lindel\"of bound is known for the {\it fourth}, {\it sixth}, and {\it eighth} moment of $ L(s_f, f) $ (see \cite{Luo-Twisted-LS,Young-GL(3)-Special-Points,Chandee-Li-GL(4)-Special-Points}), namely
\begin{align*}
	\sum_{ t_f \leqslant T } \omega_f |L (s_f, f)|^{k} \Lt_{\vepsilon} T^{2+\vepsilon}, \qquad k = 4, 6, 8,  
\end{align*}
while the author proved  (see \cite[Corollary 8]{Qi-GL(2)-Special}) that the {\it second} moment has asymptotic:
\begin{align*}
		\sum_{ t_f \leqslant T } \omega_f |L (s_f, f)|^2 =   \frac {1} {\pi^2} T^2 \log T +  \frac {2 \gamma - 2 \log 8\pi^2 -1} {2\pi^2} T^2 + O_{  \vepsilon} \big(T^{3/2+\vepsilon}\big) ,  
\end{align*}
so  there is probably a secondary term of order $T^{3/2}$ hidden in the error $ O_{  \vepsilon} \big(T^{3/2+\vepsilon}\big) $. It is reasonable to conjecture that such a secondary term  exists  for higher moments, but it is not clear to the author how to formulate the conjecture more explicitly.  

The study of moments of central $L$-values $L(1/2, f)$ for $f \in \SB_0 $ is quite a different story. For example, it has been verified that their {\it first}, {\it second}, and {\it third} moments all have error  $O (T^{1+\vepsilon})$, while the {\it fourth} moment has error $O (T^{4/3+\vepsilon})$. Further,  in the spirit of \cite{Conrey-FKRS-Moments}, Ivi\'c conjectured for the $k$-th moment the error $O (T^{1+\vepsilon})$.  The reader is referred to  \cite{Motohashi-JNT-Mean,Kuznetsov-4th-Moment,Ivic-Moment-1,Ivic-Jutila-Moments,Qi-Ivic}.

\subsection{The Explicit Formulae}   \label{sec: explicit formulae}

It seems hard for the secondary terms in the asymptotic formulae in  Theorem \ref{thm: asymp} 
to emerge from the approach in \cite{Qi-GL(2)-Special}.  
Instead, it 
will be deduced from an explicit formula for the smoothly weighted mean value of $ L(s_f, f) $ for $f \in \SB_{\delta}$; in particular, it does not rely on the approximate functional equation and the analysis for the Bessel integrals in \cite{Qi-GL(2)-Special}. 

The explicit formulae will be deduced from the Kuznetsov trace formula, the Poisson summation formula, the 
functional equation of the Lerch zeta function,  the  Kummer transformation formula
\begin{align}\label{1eq: Kummer}
	\Phi (\valpha, \gamma; z) = \exp (z) \Phi (\gamma-\valpha, \gamma; - z), 
\end{align} and the reciprocity formula
\begin{align}\label{1eq: reciprocity}
	e \Big( \!    \mp \frac {\widebar{n}} {c} \Big) = e \Big( \!    \pm \frac {\widebar{c}} {n} \Big) e \bigg( \!    \mp \frac {1} {c n} \bigg) .
\end{align}   
 It is crucial that the analytic and arithmetic formulae in \eqref{1eq: Kummer} and \eqref{1eq: reciprocity} correlate for $z = \pm 2\pi i / c n$ to exhibit the symmetry   $ c \xleftrightarrow{\ \  }   n$. This however is not needed if we were to study the mean value of the central $L$-values $L (1/2, f)$. 

For the statement of our explicit formulae, let us introduce some definitions and notations. Some of their basic properties may be found in \S \ref{sec: lemmas}. 

\begin{defn}\label{defn: Kummer}
	Define the regularized Kummer confluent hypergeometric function  
\begin{align}  
	 \breve{\Phi}  (\valpha, \gamma; z) = \Phi (\valpha, \gamma; z) - 1 = \sum_{n=1}^{\infty} \frac {(\valpha)_n z^n } {(\gamma)_n n! },   
\end{align}  
where as usual the Pochhammer symbol  
\begin{align*}
	 (\valpha)_n = \valpha (\valpha+1) \cdots (\valpha+n-1) = \frac {\Gamma (\valpha+n)} {\Gamma (\valpha)}. 
\end{align*}
Define
\begin{align}\label{1eq: Theta}
	\breve{\Theta}_{s} (z) =  \breve{\Phi}   (  1 / 2 , 1 - s ; z  ) = \sum_{n=1}^{\infty} \frac {(2n-1)!! z^n } {(2n)!! (1-s)_n   }. 
\end{align}
\end{defn}

\begin{defn}\label{defn: A(s)}
	For $\mathrm{Re} (s) = 0$, define the double series 
	\begin{align}\label{1eq: defn A}
		\breve{A}_{\delta}  (s) = \sum_{\pm} (\pm 1)^{\delta}	 \sum_{c=1}^{\infty} c^{ s}    \mathop{ \sum_{n=1}^{\infty}  }_{(n, c) = 1}     \frac { e  ( \pm     {\widebar{n}} / {c}  ) } { \sqrt{ c n} }	  \breve{\Theta}_{s}   \bigg( \!   \mp \frac {2\pi i} { c n}  \bigg)  . 
	\end{align}
\end{defn}

\begin{defn}\label{defn: L(c)}
	For $\delta = 0$ or $ 1$, let $\SX_{\delta} (c) $ denote the set of Dirichlet characters $\vchi \, (\mathrm{mod}\, c)$ such that $\vchi (-1) = (-1)^{\delta}$. Define
	\begin{align}\label{1eq: L(c)}
	L_{\delta} (c) =	\frac 1 {\sqrt{c} \varphi (c)} \, \sumd_{   \vchi      (\mathrm{mod} \, c) }  \tau (\vchi)   L   (   1 / 2 , \vchi   ), 
	\end{align}
where the superscript $\delta$ restricts the sum on $  \SX_{\delta} (c)$, as usual  $\varphi (c)$ is the Euler toitent function and $\tau (\vchi) $ is the Gauss sum.  
\end{defn}

\begin{defn}\label{defn: gamma}
	Define $\gamma$-functions 
	\begin{align}
		\gamma_{\delta}  (s) =	 { (2\pi)^{ -s }  }    \Gamma  ( s  )     \left\{ \begin{aligned}
			&   \cos   (\pi  s/2 )   , 
			\\
			&  i \sin  (\pi s/2 )  , 
		\end{aligned} \right.    \quad  \gamma_{\delta}^{\natural}  (s ) =       \frac { \Gamma (1/2-s) } {\Gamma (1-s)}  \left\{ \begin{aligned}
		&   \tan (\pi s /2)+1   , 
		\\
		&   \tan (\pi s /2) -1  , 
	\end{aligned} \right.  
	\end{align}
	 according as $\delta = 0$ or $1$. 
\end{defn}

\delete{and the double Lerch zeta function of the form: 
\begin{align}
	Z (s, w) = \mathop{\sum_{c=1}^{\infty}  \sum_{n=1}^{\infty}  }_{(c, n)  = 1} \frac {e (\widebar{n} / c)} {c^{s} n^{w}} .  
\end{align}}

\begin{theorem}\label{thm: explicit}
	Let $\phi (t)  $ be a holomorphic function on $|\operatorname{Im}(t)|\leqslant 1$ such that 
	\begin{align}\label{1eq: phi (i/2)}
		\phi (\pm i/2) = 0, 
	\end{align}
	\begin{align}\label{1eq: decay of h}
		\phi (t) \Lt \exp  ( - a |t|^2 )  , \qquad \text{{\rm($ a > 0$)}},
	\end{align} 
and, for $\delta = 0$ or $1$, define
\begin{align}
	\label{1eq: Cs(h)}
	\EuScript{C}_{\delta} (  \phi) =    \sum_{f \in \SB_{\delta}} \omega_f (L (1/2+i t_f, f) \phi  (t_f) + L (1/2-i t_f, f) \phi (- t_f) ).  
\end{align}
Then we have the explicit formulae
\begin{align} \label{1eq: C0}
	\EuScript{C}_{0} (  \phi) = \EuScript{D} (\phi) - \EuScript{E}  (  \phi) - \EuScript{E}' (  \phi) +  {\EuScript{A}}_{0}^{0} (\phi) + \breve{\EuScript{A}}_{0}^{1}  (\phi) + \breve{\EuScript{A}}_{0}^{\natural}  (\phi) + 	\EuScript{K}_{0}^{1} ( \phi) + 	\EuScript{K}_{0}^{\natural} (  \phi) , 
\end{align}
\begin{align} \label{1eq: C1}
	\EuScript{C}_{1} (  \phi) = \EuScript{D} (\phi) +  {\EuScript{A}}_{1}^{0} (\phi) + \breve{\EuScript{A}}_{1}^{1}  (\phi) + \breve{\EuScript{A}}_{1}^{\natural}  (\phi) + 	\EuScript{K}_{1}^{1} ( \phi) + 	\EuScript{K}_{1}^{\natural} (  \phi) , 
\end{align}
where 
\begin{align}\label{1eq: D(h)}
\EuScript{D} (\phi) = \frac 1 {\pi^2} \int_{-\infty}^{\infty} \phi (t)     \tanh (\pi t) t  \nd t , 
\end{align}
\begin{align}\label{1eq: E}
	\EuScript{E} (  \phi) =  \frac {2 } {\pi} \int_{-\infty}^{\infty} \frac  {  \zeta (1/2) \zeta (1/2+2it) } {\, |\zeta (1+2it)|^2 }  \phi (t) \nd t , \qquad \EuScript{E}' (  \phi) =   \frac { 2 \phi  ( - i /4 )  } { \zeta (3/2)   }, 
\end{align}
\begin{align}\label{1eq: A0}
	{\EuScript{A}}_{\delta}^{0} (\phi) =   \frac {2 i^{\delta}  } {\pi^2  }  \int_{-\infty}^{\infty}  {\phi (t)}  \gamma_{\delta} ( 1/2 - 2  i t ) \tanh  (\pi t)     t     \nd t  , 
\end{align}
\begin{align}\label{1eq: A1}
	\breve{\EuScript{A}}_{\delta}^{1} ( \phi) =     \frac {2  } {\pi^2 i^{\delta} }  
	\int_{ - \infty}^{ \infty}    {\phi (- t)} \breve{A}_{\delta}  (2it)   \gamma_{1}  (2it) t   \nd t, 
\end{align} 
\begin{align}\label{1eq: An}
	\breve{\EuScript{A}}_{\delta}^{\natural} ( \phi) =   \frac { i^{\delta} } {\pi \sqrt{\pi} i}  
	\int_{ - \infty}^{ \infty}    {\phi ( t)} \breve{A}_{\delta}  (2it)   \gamma_{\delta}^{\natural}  (2it) t   \nd t, 
\end{align} 
\begin{align}\label{1eq: K1}
	\EuScript{K}_{\delta}^{1} ( \phi) =  \frac {4  } {\pi^2 i^{\delta} }   \sum_{c=1}^{\infty}         {L_{\delta} (c)   }    
	\int_{\, i - \infty}^{i+\infty}  {c^{  2it}}  {\phi (- t)}  \gamma_{1}  (2it) t   \nd t , 
\end{align} 
\begin{align} \label{1eq: Kn}
		\EuScript{K}_{\delta}^{\natural} (  \phi) =   \frac {2 i^{\delta} } {\pi \sqrt{\pi} i}   \sum_{c=1}^{\infty}    {L_{\delta} ( c )}   
		\int_{\, i - \infty}^{i+\infty}  {c^{  2it}}  {\phi ( t)} 	\gamma_{\delta}^{\natural} (2it)  t   \nd t . 
\end{align}
\end{theorem}

The proof of Theorem \ref{thm: explicit} uses the method of regularization  in \cite{Qi-Hankel}. It was introduced to resolve the dilemma in the proofs of the formula of Kuznetsov and Motohashi on the mean square of $L (1/2, f)$ for $ f \in \SB_0 $ in the literature \cite{Kuznetsov-Motohashi-formula,Motohashi-JNT-Mean,Motohashi-Riemann}:  the Vorono\"i--Oppenheim summation formula used by Kuznetsov does not apply to the $K$-Bessel integral, while the functional equation of the Estermann function used by Motohashi does not apply to the $J$-Bessel integral. The method of regularization features uniform analyses for the Bessel integrals and synthesizes the summation formula and the functional equation.  For more discussions, the reader is referred to \S \S 1.1.2 and 7.4 in \cite{Qi-Hankel}. 

It seems that the method of regularization works well not only  for the Kuznetsov formula but also for the Petersson formula of small weight.

\section{Proof of Theorem \ref{thm: asymp}: the Asymptotic Formulae} 

The   purpose of this section is to deduce the asymptotic formulae in Theorem \ref{thm: asymp} from the explicit formulae in Theorem \ref{thm: explicit}. 

\subsection{Preparation: Simple Analytic Lemmas} \label{sec: lemmas}

\begin{lem}\label{lem: A}
	Let 
	$ 	\breve{A}_{\delta}  (s) $ be defined in Definition  
	{\ref{defn: A(s)}}. Then for $\mathrm{Re} (s) = 0$ we have 
	\begin{align}\label{5eq: bound for A}
		\breve{A}_{\delta}  (s) = O \bigg( \frac 1 {|s|+1} \bigg). 
	\end{align} 
\end{lem}

\begin{proof}
	By \eqref{1eq: Theta}, 
	\begin{align*}
		|\breve{\Theta}_{s} (z) | \leqslant \frac 1 {\sqrt{2}}  \sum_{n=1}^{\infty} \frac {|z|^n} {|s|^n + n!} .  
	\end{align*}
	Thus $ \breve{\Theta}_{s} (z) = O (|z| / (|s|+1) ) $ if $|z| \leqslant 2\pi$, and \eqref{5eq: bound for A} follows from this and the definition of $\breve{A}_{\delta}  (s)$ in \eqref{1eq: defn A}.  
\end{proof}

\begin{lem}\label{lem: L(c)}
	Let 
	$ L_{\delta} (c)  $ be defined in Definition  
	{\ref{defn: L(c)}}. Then  
	\begin{align}\label{5eq: bound for L(c)}
		L_{\delta} (c)  = O_{\vepsilon}  \big( c^{1/4+\vepsilon} \big). 
	\end{align} 
\end{lem}

\begin{proof}
	This lemma is trivial  by the convexity bound for $L (1/2, \vchi)$. Note that the non-primitive case may be easily reduced to the primitive case as in  \cite[\S  5]{Davenport-Mult-NT}. 
\end{proof}

\begin{lem}
	\label{lem: gamma}
	Let $  \gamma_{\delta} (s) $ and $\gamma_{\delta}^{\natural} (s)$ be defined in Definition {\ref{defn: gamma}}. Then for $s = \sigma \pm it$, with $\sigma$ fixed and $ t  $ large, we have  
	\begin{align}\label{5eq: gamma}
		\gamma_{\delta} (s) = O_{\sigma} \big(t^{\sigma - \frac 1 2 } \big), 
	\end{align}
	\begin{align}\label{5eq: gamma, 2}
		\gamma_{\delta}^{\natural} (s) = (\pm i)^{1 - \delta} \sqrt{\frac {2} {t} }    \bigg( 1 + O_{\sigma} \bigg(\frac 1 {t } \bigg)\bigg) . 
	\end{align}
\end{lem}

\begin{proof}
	
	This lemma is a simple consequence of    the Stirling formula (see  \cite[\S  1.1]{MO-Formulas})
	\begin{align*} 
		\log \Gamma (s) = \lp s - \frac 1 2 \rp \log s - s + \frac 1 2 \log (2\pi) + O \lp \frac 1 {|s|} \rp ,
	\end{align*}
	{or for $s = \sigma \pm it $ as above
		\begin{align*}
			\Gamma (s) = \sqrt{2\pi} (\pm it )^{\sigma - \frac 1 2 }  \exp    ( \pm i   (t \log t - t   )  -   {\pi t} / 2 )    \bigg( 1 + O_{\sigma} \bigg(\frac 1 {t } \bigg)\bigg) . 
	\end{align*} }
\end{proof}

\subsection{Setup} 

Let $T, \varPi$ be large parameters such that $T^{\vepsilon} \leqslant \varPi \leqslant T^{1-\vepsilon}$.  
Define 
\begin{align}
	{\phi}_{T, \varPi} (t) = \exp \bigg( \!   - \frac {(t-T)^2 } {\varPi^2} \bigg) .
\end{align}
Consider the smoothly weighted mean values
\begin{align}
	\EuScript{C}_{ \delta} (T, \varPi)    =    \sum_{f \in \SB_{\delta}} \omega_f  L (1/2+i t_f, f)  {\phi}_{T, \varPi}  (t_f) ,   
\end{align}
\begin{align}
	\	\EuScript{E} (T, \varPi)   = 	\frac {2   } {\pi} \int_{-\infty}^{\infty} \frac  {  \zeta (1/2) \zeta (1/2+2it) } {\, |\zeta (1+2it)|^2 }  \phi_{T, \varPi} (t) \nd t. 
\end{align}
Our aim is to prove the following asymptotic formulae. 

\begin{theorem}\label{thm: smooth}
	We have 
		\begin{align} 
			 \EuScript{C}_{ 0} (T, \varPi) + \EuScript{E} (T, \varPi)  =  \frac {\varPi T} {\pi \sqrt{\pi}} + \frac {2 \varPi \sqrt{T}} {\pi } + O \bigg( \frac {\varPi } {\sqrt{T}}  \bigg(1 + \frac {\varPi^2} {T} \bigg)\bigg)  ,
	\end{align} 
	\begin{align}  	\EuScript{C}_{ 1} (T, \varPi) 
		=  \frac {\varPi T} {\pi \sqrt{\pi}}   + O \bigg( \frac {\varPi } {\sqrt{T}}  \bigg(1 + \frac {\varPi^2} {T} \bigg)\bigg)  .   
	\end{align} 
\end{theorem}

\subsection{Proof of Theorem \ref{thm: smooth}} 

First of all, in order to apply Theorem \ref{thm: explicit}, we introduce the modified weight 
\begin{align}
	\check{\phi}_{T, \varPi} (t) = \frac {t^2+1/4} {t^2+ 4} \exp \bigg( \!   - \frac {(t-T)^2 } {\varPi^2} \bigg), 
\end{align}
so that \eqref{1eq: phi (i/2)} and \eqref{1eq: decay of h} are satisfied. 
Note that both $ 	\phi_{T, \varPi} (t)  $ and $\check{\phi}_{T, \varPi} (t)$ are exponentially small outside   $$ |\mathrm{Re} (t) - T| \leqslant \varPi \log T ,  $$ uniformly for $ |\mathrm{Im} (t)| \leqslant 1$, while 
\begin{align}\label{6eq: phi's}
	\check{\phi}_{T, \varPi} (t) = {\phi}_{T, \varPi} (t) + O \big(1/T^2 \big) 
\end{align}
on this range. From this and the lemmas in \S \ref{sec: lemmas}, we transform the explicit formulae in \eqref{1eq: C0} and \eqref{1eq: C1} into the following asymptotic formulae. 

\begin{proposition} We have
	\begin{align} 
		 	\EuScript{C}_{ 0} (T, \varPi) + \EuScript{E} (T, \varPi) = {\EuScript{D}}  (T, \varPi) +  	\EuScript{K}_{0}^{\natural} ( {T, \varPi}) + O \bigg(\frac {\varPi} {\sqrt{T}}  \bigg) , 
	\end{align} 
	\begin{align}  	\EuScript{C}_{ 1} (T, \varPi)  = {\EuScript{D}}  (T, \varPi) +  	\EuScript{K}_{1}^{\natural} ( {T, \varPi}) + O \bigg(\frac {\varPi} {\sqrt{T}}  \bigg) , 
\end{align} 
where 
\begin{align}
	{\EuScript{D}}  (T, \varPi) = \frac 1 {\pi^2} \int_{-\infty}^{\infty} \phi_{T, \varPi} (t) t  \nd t,  
\end{align}
\begin{align}\label{2eq: K}
	 \EuScript{K}_{\delta}^{\natural} ( {T, \varPi}) =  \frac {2  } {\pi \sqrt{\pi}  }   \sum_{c=1}^{\infty}    \frac {L_{\delta} ( c )}   {c^2}
	\int_{ - \infty}^{\infty}  {c^{  2it}}  {\phi_{T, \varPi} (  i + t)}    {\sqrt{t}}    \nd t . 
\end{align}
\end{proposition}

\begin{proof}
	It follows from \eqref{6eq: phi's} and the convexity bounds that
	\begin{align*}
		\EuScript{C}_{\delta} (T, \varPi) = \EuScript{C}_{\delta} (\check{ \phi}_{T, \varPi}) + O  \bigg( \frac {\varPi T^{\vepsilon}} {T^{3/4}}\bigg)  , \quad \EuScript{E} (T, \varPi) = \EuScript{E} (\check{ \phi}_{T, \varPi}) + O  \bigg( \frac {\varPi T^{\vepsilon}} {T^{7/4 }} \bigg)  . 
	\end{align*}
	It is also clear that  
	\begin{align*}
		\EuScript{D} (\check{\phi}_{T, \varPi}) = 
		\frac 1 {\pi^2} \int_{-\infty}^{\infty} \phi_{T, \varPi} (t) t  \nd t + O \lp \frac {\varPi} {T} \rp, \qquad \EuScript{E}' (  \check{\phi}_{T, \varPi}) \Lt \exp \lp - \frac {T^2} {\varPi^2} \rp    . 
	\end{align*} 
	By contour shift to $\mathrm{Re} (it) = 1$,   it follows from \eqref{5eq: gamma} in Lemma \ref{lem: gamma}  that
	\begin{align*}
		{\EuScript{A}}_{\delta}^{0} (\check{\phi}_{T, \varPi} ) \Lt \frac {\varPi} {T}. 
	\end{align*} 
	By  Lemmas \ref{lem: A} and \ref{lem: gamma}, we infer that
	\begin{align*}
		\breve{\EuScript{A}}_{\delta}^{1} (\check{ \phi}_{T, \varPi}), \, \breve{\EuScript{A}}_{\delta}^{\natural} (\check{\phi}_{T, \varPi}) \Lt \frac {\varPi} {\sqrt{T}} . 
	\end{align*}
	By Lemmas \ref{lem: L(c)} and \ref{lem: gamma},  
	\begin{align*}
		\EuScript{K}_{\delta}^{1}  (\check{ \phi}_{T, \varPi})  \Lt \frac {\varPi} {T \sqrt{T} } . 
	\end{align*}
By \eqref{5eq: gamma, 2} in Lemma  \ref{lem: gamma}, along with \eqref{6eq: phi's}, we deduce that
\begin{align*}
	\EuScript{K}_{\delta}^{\natural}  (\check{ \phi}_{T, \varPi}) = 	\frac {2  } {\pi \sqrt{\pi}  }   \sum_{c=1}^{\infty}    \frac {L_{\delta} ( c )}   {c^2}
	\int_{ - \infty}^{\infty}  {c^{  2it}}  {\phi_{T, \varPi} (  i + t)}    {\sqrt{t}}    \nd t + O \lp \frac {\varPi } {\sqrt{T}} \rp .  
\end{align*}
Thus the proof is completed if we combine the above bounds and asymptotics. 
\end{proof}

Clearly the main term is from an evaluation of the integral in $ \EuScript{D} ( {T, \varPi}) $
\begin{align*}
\int_{-\infty}^{\infty} \phi_{T, \varPi} (t) t  \nd t = \sqrt{\pi} {\varPi T} ,
\end{align*}
by the change of variable $t \ra T + \varPi t$. 

As for $ \EuScript{K}^{\natural}_{\delta} ( {T, \varPi}) $, since the integral in \eqref{2eq: K} is  Fourier of phase $2 t \log c$, the  contribution is negligibly small if $c > 1$ (so that $\log c \Gt 1$). However, 
\begin{align*}
	L_{\delta} (1) = 1 - \delta,  
\end{align*} 
as $  \SX_{0} (1) = \{ \boldsymbol{1} \}$ and $\SX_{0} (1) = \text{\O}   $. 
Therefore the secondary term arises in the case $\delta = 0$ and follows from the calculations
\begin{align*}
	 	\int_{ - \infty}^{\infty}     {\phi_{T, \varPi} (  i + t)}    {\sqrt{t}}    \nd t & = \int_{ - \infty}^{\infty}     {\phi_{T, \varPi} ( t)}    {\sqrt{t}}     \nd t + O \lp \frac {\varPi } {\sqrt{T}} \rp \\
	 	& = \sqrt{\pi} \varPi \sqrt{T} + O \bigg( \frac {\varPi } {\sqrt{T}}  \bigg(1 + \frac {\varPi^2} {T} \bigg)\bigg),
\end{align*}
by contour shift, variable change, and Taylor approximation.

\subsection{Removal of the Smooth Weight} 
Let  $T^{\vepsilon} \leqslant   H \leqslant T / 3$. 	Define 
\begin{align}
	\acute{\EuScript{C}}_{ \delta}  (T, H)  & =    \sumd_{|t-T| \leqslant H } \omega_f  L (1/2+i t_f, f)    , \\
	\acute{\EuScript{E}} (T, H) & = 	\frac {2    } {\pi} \int_{T-H}^{T+H} \frac  {\zeta (1/2) \zeta (1/2+2it) } {\, |\zeta (1+2it)|^2 }   \nd t. 
 \end{align}

\begin{lem}\label{lem: unsmooth}
	Let  $T^{\vepsilon} \leqslant \varPi^{1+\vepsilon}  \leqslant H \leqslant T / 3$.  
	Then
	\begin{align}\label{6eq: Cs=C}
	\acute{\EuScript{C}}_{ \delta}  (T, H) & = \frac 1 {\sqrt{\pi} \varPi}  \int_{T-H}^{T+H} \EuScript{C}_{ \delta} (K, \varPi) \nd K + O \big( \varPi T^{1+\vepsilon} \big),  \\
		\label{6eq: Es=E}
		\acute{\EuScript{E}}  (T, H) & = \frac 1 {\sqrt{\pi} \varPi}  \int_{T-H}^{T+H} \EuScript{E} (K, \varPi) \nd K + O \big( \varPi T^{1/4+\vepsilon} \big) . 
	\end{align}
\end{lem} 

\begin{proof}
	The asymptotics in \eqref{6eq: Cs=C} and \eqref{6eq: Es=E} follow  if we slightly modify the proofs and statements of Lemmas 5.3 and 5.4 in \S 5 of \cite{Qi-Liu-Moments}  (therein we considered the central $L$-values $L(1/2, f)$). 
	
	For \eqref{6eq: Cs=C}, we need the mean-Lindel\"of bound for the mean value of $ | L (1/2+i t_f, f) |  $ on $ |t_f - T| \leqslant \varPi $,  
	which,  by Cauchy--Schwarz,  may be deduced from  
	\begin{align*}
		\sumd_{|t-T| \leqslant \varPi } \omega_f  \Lt \varPi T, \qquad 	 \sumd_{|t-T| \leqslant \varPi } \omega_f  |L (1/2+i t_f, f) |^2   \Lt \varPi T \log T  ; 
	\end{align*}  
	see \cite[Corollary 1.2]{XLi-Weyl} and \cite[Corollary 7]{Qi-GL(2)-Special}.  
	
	For \eqref{6eq: Es=E}, the $T^{1/4+\vepsilon}$ in the error term $O  ( \varPi T^{1/4+\vepsilon}  )$ is due to the convexity bound for $ |\zeta (1/2+2it) |$. 
\end{proof}

\begin{lem}\label{lem: bound for zeta}
	We have
	\begin{align}\label{6eq: Es} 
		\acute{\EuScript{E}} (T, H)  =  O \big(  H^{1/2} T^{1/2+\vepsilon}  \big) . 
	\end{align}
\end{lem} 

\begin{proof}
	By Cauchy--Schwarz, this (crude) bound follows from 
	\begin{align*}
		    \frac {1}  {\zeta (1+it)} = O (\log |t|), \qquad \int_{0}^{T} |\zeta (1/2+it)|^2 \nd t = O (T \log T); 
	\end{align*}
see \cite[\S \S 3.11, 7.3]{Titchmarsh-Riemann}.
\end{proof}

Then we have the following corollary of Theorem \ref{thm: smooth} and Lemmas \ref{lem: unsmooth}, \ref{lem: bound for zeta} on the choice $\varPi = T^{\vepsilon}$.

\begin{coro} Let  $T^{\vepsilon} \leqslant   H \leqslant T / 3$.  We have
	\begin{align}
		\sideset{}{^{0}}\sum_{|t-T| \leqslant H } \omega_f  L (1/2+i t_f, f)   = \frac {2 H T} {\pi^2 } + \frac {4 \big((T+H)^{3/2} - (T-H)^{3/2} \big)} {3 \pi^{3/2} } + O \big(T^{1+\vepsilon} \big), 
	\end{align}
	\begin{align}
	\sideset{}{^{1}}\sum_{|t-T| \leqslant H } \omega_f  L (1/2+i t_f, f)   = \frac {2 H T} {\pi^2 }  + O \big(T^{1+\vepsilon} \big). 
\end{align}
\end{coro}

Finally, Theorem \ref{thm: asymp} is a direct consequence of the above asymptotic formulae with $H = T/3$ along with a dyadic summation.

	\section{Preliminaries}

\subsection{Kloosterman Sums} Recall that 
\begin{align}\label{3eq: Kloosterman}
	S   (m, n ; c ) = \sumx_{   a      (\mathrm{mod} \, c) } e \bigg(   \frac {  a     m +   \widebar{a    } n} {c} \bigg) ,
\end{align} 
where the $\star$ indicates the condition $(a    , c) = 1$ and $ \widebar{a    }$ is given by $a     \widebar{a    }  \equiv 1 (\mathrm{mod} \, c)$. 
We have the Weil bound: 
\begin{align}
	\label{2eq: Weil}
	S (m, n; c)   \Lt \tau (c) \sqrt{ (m, n, c) } \sqrt{c} ,  
\end{align} 
where as usual $\tau (c)  $ is the number of divisors of $c$. 

\subsection{Bessel Functions} 
Let $J_{\vnu} (z)$, $ I_{\vnu} (z)$, and $K_{\vnu} (z)$ be the Bessel functions as in \cite{Watson}.
Recall that $J_{\vnu} (z)$ and $I_{\vnu} (z)$ are defined by the series
\begin{align}\label{2eq: defn of J, I}
	J_{\vnu} (z) = \sum_{n=0}^{\infty} \frac {(-1)^n (z/2)^{\vnu +2n}} {n! \Gamma (\vnu + n + 1)},  \qquad 
	I_{\vnu} (z) = \sum_{n=0}^{\infty} \frac {(z/2)^{\vnu +2n}} {n! \Gamma (\vnu + n + 1)}. 
\end{align}
By \cite[3.7 (6)]{Watson},  
\begin{align}\label{2eq: K=I}
	 K_{\vnu} (z) = \frac {\pi}   { 2 \sin (\pi \vnu) }  ( {I_{-\vnu} (z) - I_{\vnu} (z)} ) . 
\end{align}
In order to unify our later analysis, for $x \in \mathbf{R}_{+}$ let us define 
\begin{align}\label{2eq: B+-(x)}
	B_{\vnu}^{+} (x) =  \frac {\pi}   {  2 \sin (\pi \vnu) } (  J_{-2\vnu} (x) - J_{2\vnu} (x)  ), \qquad B_{\vnu}^{-} (x) = 2 \cos (\pi\vnu) K_{2\vnu} (x) . 
\end{align} 

\subsection{Kuznetsov Trace Formulae} 

Let  $h (t)$ be an even function satisfying the conditions{\hspace{0.5pt}\rm:}
\begin{enumerate} 
	\item[{\rm (i)\,}] $h (t)$ is holomorphic in  $|\operatorname{Im}(t)|\leqslant {1}/{2}+\vepsilon$,
	\item[{\rm (ii)}] $h(t)\Lt (|t|+1)^{-2-\vepsilon}$ in the above strip. 
\end{enumerate}	
For $m, n \geqslant 1 $, define the cuspidal
\begin{align}
	\Delta_{ \delta} (m, n) = \sum_{f \in \SB_{\delta}}  \omega_f h(t_f) \lambda_f ( m )  { \lambda_f ( n) } ,
\end{align}
and the Eisenstein 
\begin{align}
	\Xi (m, n )	= \frac{1}{\pi}   \int_{-\infty}^{\infty} \omega(t) h(t)  \tau_{it}(m) \tau_{it} (n) \nd t, 
\end{align}
where $\omega_f$ and $\lambda_{f} (n)$ are defined before in \S \ref{sec: intro}, and
\begin{align}\label{2eq: divisor}
	\omega (t) = \frac {1 } {|\zeta (1 + 2it)|^2} , \qquad 	\tau_{\vnu} (n) = 
	\sum_{ a b = n} (a/b)^{  \vnu} .  
\end{align}
On the other hand, define the Plancherel integral
\begin{align}\label{app: integral H}
	H   = \frac {1} {  \pi^2} \int_{-\infty}^{\infty} h(t)    {  \tanh (\pi t) t  \nd t }    , 
\end{align}
and the Kloosterman--Bessel sum 
\begin{align}\label{app: KB +-}
	\mathrm{KB}^{  \pm} (m, n ) =  \sum_{  c = 1 }^{\infty}  \frac {S (  m , \pm n ; c)} { c}  H^{  \pm} \bigg(    \frac {4\pi \textstyle \sqrt { m n } } { c} \bigg),
\end{align} 
where $S (m, \pm n; c)$ is the Kloosterman sum and  $H^{\pm} (x)$ is the Bessel integral 
\begin{align} \label{app: H + (x)}
	 & H^{ \pm } (x) = \frac 2 {\pi^2  }   \int_{-\infty}^{\infty}   h(t) B_{it}^{\pm} (  x)   {\tanh (\pi t)} t   \nd t .
\end{align}
By \eqref{2eq: K=I} and \eqref{2eq: B+-(x)}, for $h(t)$ even we have
\begin{align}\label{2eq: H+-(x), 2}
	H^{+} (x) =   \frac {2 i} {\pi} \int_{-\infty}^{\infty}   h(t) J_{2it} ( x) \frac {t   \nd t  } {\cosh (\pi t)}, \quad H^{-} (x) =   \frac {2 i} {\pi} \int_{-\infty}^{\infty}   h(t) I_{2it} ( x) \frac {t   \nd t  } {\cosh (\pi t)}. 
\end{align}
Thus the convergence of the $ c $-sum in \eqref{app: KB +-} follows from the Weil bound \eqref{2eq: Weil} and a contour shift of the Bessel integrals in \eqref{2eq: H+-(x), 2}   to $ \mathrm{Re} (it) = 3/8$ say.  

The Kuznetsov trace formula for the even or odd cusp forms reads (see \cite[\S 3]{CI-Cubic}): 
\begin{align}\label{2eq: Kuznetsov}
	2	\Delta_{  \delta }  (m, n )  + 2 (1-\delta) \shskip \Xi   (m, n) =    \delta (m, n) \shskip H  +  \mathrm{KB}^{  +}  (m, n) + (-1)^{\delta} \mathrm{KB}^{  -}  (m, n) ,
\end{align}
where   $\delta (m, n)$ is the Kronecker $\delta$-symbol. 

\subsection{Poisson Summation Formula} 

For $f (x) \in L^1 (\mathbf{R})$ define its Fourier transform: 
\begin{align*}
	\hat {f} (y) =	\int_{-\infty}^{\infty} f (x) e  ( - x y)   \nd x .
\end{align*}
The following Poisson summation formula is a special case of 
\cite[Corollary VII 2.6]{Stein-Weiss}. 
\begin{lem}\label{lem: Poisson}   Suppose that
	\begin{align}\label{4eq: conditions}
		f (x) = O \lp \frac 1  {(|x|+1)^{1 + \vepsilon}} \rp , \qquad \hat{f} (y) = O \lp \frac 1  {(|y|+1)^{1 + \vepsilon}} \rp , 
	\end{align}
	for $\vepsilon > 0${\rm;} thus both $ f $ and $\hat{f}$ can be assumed to be continuous. Then 
	\begin{align}\label{2eq: Poisson}
		\sum_{n = - \infty}^{\infty} e (x n ) f (n) & =   \sum_{n = - \infty}^{\infty}   \hat {f}  ( n - x ) .
	\end{align} 
\end{lem}

\subsection{The Lerch $\zeta$-function}

For $0 < w \leqslant 1$ and real $x$, define the Lerch zeta function 
\begin{align}\label{2eq: Lerch}
	\zeta (s, w, x) = \sum_{n=0}^{\infty}  \frac {e (x n)} {(n+w)^{s}}, 
\end{align}
for  $\mathrm{Re} (s) > 1$ if $x$ is an integer or $\mathrm{Re} (s) > 0$ if otherwise. Note that when $x$ is  integral, it is reduced to the Hurwitz zeta function 
\begin{align}\label{2eq: Hurwitz}
	\zeta (s, w) = \sum_{n=0}^{\infty}  \frac {1} {(n+w)^{s}}.
\end{align} 
Furthermore, if $w = 1$, $\zeta (s, 1) = \zeta (s)$, the Riemann zeta function. 

It is well-known that $	\zeta (s, w, x)$ has analytic continuation on the entire complex plane except for a simple pole at $s = 1$ of residue $1$ in the case of Hurwitz $ \zeta (s, w)$.  
For $0 < x < 1$, we have the functional equation (the Lerch transformation formula) \cite{Lerch}: 
\begin{equation*}
		\zeta (1-s, w, x) = \frac {\Gamma (s)} {(2\pi)^s}   \bigg\{ e \lp \frac {s} {4} - x w \rp \zeta (s, x, - w)  + e \lp - \frac {s} {4}  + w (1-x)  \rp \zeta (s, 1-x,  w) \bigg \}. 
\end{equation*}
For our application, we only need the special case:
\begin{equation}\label{2eq: Lerch, FE}
	e(x) \zeta (1-s, 1, x) = \frac {\Gamma (s)}    {(2\pi)^s}   \bigg\{ e \lp \frac {s} {4}   \rp \zeta (s, x)  + e \lp - \frac {s} {4}    \rp \zeta (s, 1-x ) \bigg \}. 
\end{equation}
Moreover, let us record here the Riemann functional equation:
\begin{equation}\label{2eq: Reimann, FE}
	\zeta (1-s) = 2 \frac {  \Gamma (s)} {(2\pi)^s}  \cos \lp \frac {\pi s} 2 \rp \zeta (s ) . 
\end{equation}

\subsection*{Convention} For $s \in \mathbf{C}$ and $x \in \mathbf{R}_+$ let us set  
\begin{align}\label{2eq: s power}
	(\pm i x)^s = e (\pm s/ 4) x^s. 
\end{align}
Given this, for $\mathrm{Re} (s) > 1$ we may reformulate \eqref{2eq: Lerch, FE} as follows:
\begin{equation}\label{2eq: Lerch, FE, 2}
e(x) 	\zeta (1-s, 1, x) =   {\Gamma (s)}   \sum_{n = - \infty}^{\infty} \frac 1 {(2\pi i (n- x))^s}, \qquad \text{($0 < x < 1$)} . 
\end{equation}
Thus the Poisson summation formula  \eqref{2eq: Poisson} may be regarded as an incarnation of the functional equation \eqref{2eq: Lerch, FE}. See for example  \cite{Lagarias-Li-Lerch-I}. Similarly, for $\mathrm{Re} (s) > 1$  we rewrite \eqref{2eq: Reimann, FE} as 
\begin{equation}\label{2eq: Riemann, FE, 2}
	\zeta (1-s ) =   {\Gamma (s)}   \sum_{n = 1}^{\infty} \lp  \frac 1 {(2\pi i n)^s} +  \frac 1 {(- 2\pi i n)^s} \rp . 
\end{equation}

\subsection{Dirichlet $L$-functions}  For  a Dirichlet character $\vchi \, (\mathrm{mod}\, c)$ (not necessarily primitive), define 
\begin{align}
	L (s, \vchi) = \sum_{n=1}^{\infty}  \frac {\vchi (n)} {n^{s}} , 
\end{align}
for $\mathrm{Re} (s) > 1$, and by analytic continuation on the entire complex plane. Hurwitz used the identity
\begin{align*}
	L (s, \vchi) = \frac 1 {c^s} \sum_{a = 1}^{c} \vchi (a) \zeta \Big(s, \frac a c \Big) 
\end{align*}  
to study Dirichlet $L$-functions via his $\zeta$-function (see  \cite[\S 12]{Apostol} and \cite[\S 9]{Davenport-Mult-NT}).  For our application, however, we need the  inverse
\begin{align}\label{3eq: Lerch-Dirichlet}
	\zeta \Big(s, \frac a c \Big)  = \frac {c^s} {\varphi (c)} \sum_{   \vchi      (\mathrm{mod} \, c) } \widebar{\vchi} (a) L (s, \vchi),
\end{align}
which may be proven easily by the orthogonality of characters. 

\delete{\subsection{Gauss Sums} For a Dirichlet character $\vchi \, (\mathrm{mod}\, c)$, define the associated Gauss sum
\begin{align}
	\tau (\vchi) = \sum_{   a      (\mathrm{mod} \, c) } \vchi (a) e \Big(   \frac {  a } {c} \Big). 
\end{align}}
\delete{By orthogonality, for $(a, c) = 1$ we have 
\begin{align}
	e \Big(   \frac {  a } {c} \Big) = \frac 1 {\varphi (c)} \sum_{   \vchi      (\mathrm{mod} \, c) } \widebar{\vchi} (a) \tau (\vchi). 
\end{align} }

\delete{For  $0 < \valpha \leqslant 1$, we define the Hurwitz zeta function
\begin{align}
	\zeta (s, \valpha) = \sum_{n=0}^{\infty} \frac 1 {(n+\valpha)^s},  
\end{align}
for $\mathrm{Re} (s) > 1$. 
According to Theorems 12.4 and 12.6 in \cite{Apostol},   $\zeta (s, \valpha)$ admits analytic continuation to the whole complex plane except for a simple pole at $s = 1$ with residue $1$, and we have the functional equation (the Hurwitz formula):
\begin{align}\label{2eq: Hurwitz formula}
	\zeta (1-s, \valpha) = \frac {\Gamma (s)} {(2\pi)^s} \big(e (-s/4) F (\valpha, s) + e (s/4) F (- \valpha, s) \big), \qquad \mathrm{Re}(s) > 1,
\end{align}
where 
\begin{align}\label{2eq: defn of F(b, s)}
	F (x, s) = \sum_{n=1}^{\infty} \frac {e(n x)} {n^{s}} ,
\end{align}
for $x$ real and $\mathrm{Re}(s) > 1$.}

\section{More on Bessel Functions} 

\subsection{Bounds for Bessel Functions}
Some crude uniform bounds will be needed to justify the convergence of certain infinite series and integrals. 

It follows from the Poisson integral formula in \cite[3.3 (1)]{Watson} that
\begin{align}\label{2eq: bound, 1}
	|J_{\vnu}  (x) |   \leqslant \bigg|\frac {\sqrt{\pi} (x/2)^{\vnu}   } {\Gamma (\vnu + 1/2)} \bigg|, \qquad |I_{\vnu}  (x)| \leqslant \bigg|\frac {\sqrt{\pi} (x/2)^{\vnu} \exp (x)  } {\Gamma (\vnu + 1/2)} \bigg|, 
\end{align}
if $\mathrm{Re}(\vnu) \geqslant 0$. According to \cite[\S 7.13]{Olver}, we have uniformly
\begin{align}\label{2eq: J(x), 2}
	& |J_{\vnu} (x) | \leqslant  \frac   { \cosh (\pi \mathrm{Im} (\vnu)/2)} {\sqrt{\pi x/ 2} } \bigg\{ 1 + \frac {|4\vnu^2-1|} {4x} \exp \bigg( \frac {|4\vnu^2-1|} {4x}  \bigg)\bigg\}, \\
	\label{2eq: K(x), 2} & |K_{\vnu} (x) | \leqslant \frac   {\exp (- x)} {\sqrt{\pi x/2} } \bigg\{ 1 + \frac {|4\vnu^2-1|} {4x} \exp \bigg( \frac {|4\vnu^2-1|} {4x}   \bigg)\bigg\} .
\end{align}
For $x \Gt 1$, it follows from \eqref{2eq: B+-(x)}, \eqref{2eq: J(x), 2}, and \eqref{2eq: K(x), 2} that
\begin{align}\label{2eq: B(x), decay}
	B_{\vnu}^{+} (x) = O_{\vnu} \bigg(\frac 1 {\sqrt{x}}\bigg), \qquad B_{\vnu}^{-} (x) = O_{\vnu} \bigg(\frac {\exp (-x)} {\sqrt{x}}\bigg) . 
\end{align}

\subsection{Regularized Bessel Functions} \label{sec: regularized Bessel}
Define \begin{align}\label{3eq: A(x)}
	P_{\vnu} (x) = \frac {(x/2)^{\vnu}}  {\Gamma (\vnu + 1)},   
\end{align}
and
\begin{align}\label{3eq: R(x)}
	R_{\vnu} (x) = \frac {\pi}   {  2 \sin (\pi \vnu) } (  P_{-2\vnu} (x) - P_{2\vnu} (x)  ). 
\end{align}
Note that $P_{\vnu} (x)  $ is the leading term in the expansions of $J_{\vnu} (x) $ and $I_{\vnu} (x) $ as in \eqref{2eq: defn of J, I}.  Subsequently, let  $\varww (x) \in C^\infty(\BR_+)$ be fixed such that  $\varww (x) \equiv 1$ for $x$ near $0$ and $\varww (x) \equiv  0$ for $x$ large. Define the regularized Bessel function $M_{\vnu}^{\pm} (x) $   by the difference
\begin{align}\label{3eq: M(x)}
	M_{\vnu}^{\pm} (x) = B_{\vnu}^{\pm} (x) - \varww (x^2/4 )  R_{\vnu} (x) .   
\end{align}
For $x \Lt 1$, the decay is fast: 
\begin{align}\label{3eq: M+-(x), decay}
	M_{\vnu}^{\pm} (x) = O_{\vnu}  (x^2) ,
\end{align}
if $\mathrm{Re} (\vnu) = 0$ and $\vnu \neq 0$. 

For later use, we record here the Mellin inversion formula
\begin{align}\label{3eq: Mellin inversion}
	\varww (x) = \frac 1 {2\pi i}  \int_{\sigma - i\infty}^{\sigma + i \infty} 
 \widetilde{\varww} (s) x^{- s } \nd s, \qquad \text{($\sigma > 0$)},  
\end{align}
where 
$\widetilde{\varww} (s) $ is the Mellin transform
\begin{align}\label{3eq: Mellin}
	\widetilde{\varww} (s) = \int_0^{\infty} \varww (x) x^{s-1} \nd x. 
\end{align}
For $\mathrm{Re} (s) > 0$, by partial integration,
\begin{align*}
	 \widetilde{\varww} (s) = - \frac 1 {s}  \int_0^{\infty} \varww' (x) x^{s} \nd x.
\end{align*}
Since $ \varww' (x) \in C_c^{\infty} (\mathbf{R}_{+}) $, it follows that $ \widetilde{\varww} (s)$ is analytic on $\mathbf{C}$ except for a simple pole at $s = 0$ with residue $1$ and decays rapidly on vertical strips. 

In the next two sub-sections, we shall consider the Fourier integrals
\begin{align}\label{3eq: C(y)}
	A^{\pm}_{s, \vnu} (y) & = 	\int_0^{\infty}   B^{\pm}_{ \vnu} (2 \sqrt{x}) \exp (-  i x y)  \frac { \nd      x} {x^{s+\vnu} }  , \\
	\label{3eq: N(y)}  N^{\pm}_{s, \vnu} (y) & = 	\int_0^{\infty}   M^{\pm}_{ \vnu} (2 \sqrt{x}) \exp (-  i x y)  \frac { \nd      x} {x^{s+\vnu} }, 
\end{align}
  which are functions in $C (\mathbf{R})$ as long as the integrals are  absolute convergent.  

\begin{remark}
	In the literature, the authors usually resort directly to the Mellin--Barnes integral representation of the Bessel functions. For the case of holomorphic cusp forms, it works very well  for the $J_{k-1} (x)$ in the Petersson trace formula. See for example  \cite{BF-Non-vanishing,BF-Mean-Sym2}. For the case of Maass cusp forms, it was used by Motohashi for $K_{2it} (x)$ but {not} for $J_{2it} (x)$ in the Kuznetsov trace formula. The technical difficulty is explained in \cite[\S 3.3]{Motohashi-Riemann} (see also  \cite[\S 7.4]{Qi-Hankel}). However, in contrast, it will be much easier to handle the Mellin integral in \eqref{3eq: Mellin inversion}.  
\end{remark}

\subsection{The Fourier Transform of Bessel Functions} 
\label{sec: Fourier}

The Fourier transform of Bessel functions involves the Kummer confluent hypergeometric functions $ \Phi (\valpha, \gamma; z) $ and $\Psi (\valpha, \gamma; z)$. For the reference, we shall use Chapter VI of \cite{Erdelyi-HTF-1}.

By definition,  
\begin{align}\label{3eq: defn Phi}
	\Phi (\valpha, \gamma ; z) = \sum_{n=0}^{\infty} \frac {(\valpha)_n z^n } {(\gamma)_n n! } = \frac {\Gamma (\gamma)} {\Gamma (\valpha)} \sum_{n=0}^{\infty} \frac {\Gamma (\valpha+n) z^n } { \Gamma (\gamma+n) n! } . 
\end{align}
According to \cite[6.5 (7)]{Erdelyi-HTF-1},  we have 
\begin{align}\label{Aeq: Psi}
	\Psi (\valpha, \gamma; z) = \frac {\Gamma (1-\gamma)} {\Gamma (1+ \valpha-\gamma )} \Phi (\valpha, \gamma; z) + \frac {\Gamma (\gamma-1)} {\Gamma (\valpha)} z^{1-\gamma} \Phi (1+\valpha-\gamma, 2-\gamma; z) . 
\end{align}

By \cite[6.10 (8)]{Erdelyi-HTF-1}, for $\mathrm{Re} (z) > 0$, we have 
\begin{align*}
	\int_0^{\infty} J_{\vnu} (2 \sqrt{x}) \exp (-   x z) x^{  \valpha - \frac 1 2 \vnu - 1} \nd      x =  \frac {\Gamma (\valpha)} {\Gamma (1+\vnu) } z^{- \valpha} \Phi (\valpha, 1+\vnu; -1/z) , 
\end{align*}
if $ \mathrm{Re} (\valpha) > 0 $.  
For real $y \neq 0$, if we let $ z \ra i y $,  
then 
\begin{align}\label{Aeq: J, 2}
	\int_0^{\infty} J_{\vnu} (2 \sqrt{x}) \exp (-  i x y) x^{  \valpha - \frac 1 2 \vnu- 1} \nd      x =  \frac {\Gamma (\valpha)} {\Gamma (1+\vnu) } (iy)^{- \valpha} \Phi (\valpha, 1+\vnu; i/y) , 
\end{align}
for $ 0 < \mathrm{Re} (\valpha) < \mathrm{Re} (\vnu/2) +   1 / 4$ due to the decay of $ J_{\vnu} (x) = O (1/\sqrt{x}) $ as $x \ra \infty$, where the power $ (iy)^{-\valpha} $ is defined according to \eqref{2eq: s power}.   

\begin{remark}
	One may extend the validity of {\rm\eqref{Aeq: J, 2}} to $ 0 < \mathrm{Re} (\valpha) < \mathrm{Re} (\vnu/2) +   5 / 4$ by the asymptotic expansion for $J_{\vnu} (x) $,   partial integration {\rm(}twice!{\rm)}, and analytic continuation. 
\end{remark} 

Similar to \eqref{Aeq: J, 2}, it follows from \cite[6.10  (9)]{Erdelyi-HTF-1} that 
\begin{align}\label{Aeq: K, 2}
2	\int_0^{\infty} K_{\vnu} (2 \sqrt{x}) \exp (- i x y) x^{\valpha - \frac 1 2 \vnu - 1} \nd      x =    \Gamma (\valpha) \Gamma (\valpha-\vnu) (iy)^{- \valpha} \Psi (\valpha, 1+\vnu; -i/ y), 
\end{align}
for $\mathrm{Re} (\valpha) >  \max  \{0 , \mathrm{Re}  (\vnu)  \} $.  This case is  simpler as $ K_{\nu} (x) $ is of exponential decay.

By continuity, the identities \eqref{Aeq: J, 2} and \eqref{Aeq: K, 2} are valid for $y = 0 $ if  the right-hand sides were passed to the limit by the asymptotic formulae in \cite[6.13.1 (1), (2)]{Erdelyi-HTF-1}.   By  \cite[7.7.3 (19), (27)]{Erdelyi-HTF-2},  however, it is already known explicitly that
\begin{align}\label{Aeq: J, 3}
	\int_0^{\infty} J_{\vnu} (2 \sqrt{x}) x^{  \valpha - \frac 1 2 \vnu- 1} \nd      x =  \frac {\Gamma (\valpha)} {\Gamma (1-\valpha+\vnu) } ,
\end{align}
for $ 0 < \mathrm{Re} (\valpha) < \mathrm{Re} (\vnu/2) +   3 / 4$, and
\begin{align}\label{Aeq: K, 3}
2	\int_0^{\infty} K_{\vnu} (2 \sqrt{x}) x^{  \valpha - \frac 1 2 \vnu- 1} \nd      x =  \Gamma (\valpha) \Gamma (\valpha-\vnu),
\end{align}
for $\mathrm{Re} (\valpha) >  \max  \{0 , \mathrm{Re}  (\vnu)  \} $.

Let us always assume $\mathrm{Re}(\vnu) =0$ and $\vnu \neq 0$ in what follows. 

The next two lemmas may be proven easily by the definitions or identities in \eqref{2eq: B+-(x)}, \eqref{3eq: C(y)}, \eqref{Aeq: Psi}--\eqref{Aeq: K, 3}; for  the $+$ case, one also needs to apply the Euler  reflection formula $$\Gamma (s) \Gamma (1-s) = \frac {\pi} {\sin (\pi s)}. $$ 


\begin{lem}\label{lem: C(y)}
	Define
	\begin{align}\label{3eq: Gamma}
		\Gamma (\valpha, \beta ) =	   \Gamma (\valpha) \Gamma (\valpha - \beta) , 
	\end{align}
	\begin{align}\label{3eq: Psi+-}
		\Psi_{\pm}  (\valpha, \gamma; z) \!   = \!   \frac {\Gamma (1-\gamma)} {\Gamma (1+ \valpha-\gamma )} \Phi (\valpha, \gamma; \mp z) + \frac {\Gamma (\gamma-1)} {\Gamma (\valpha)} z^{1-\gamma} \Phi (1+\valpha-\gamma, 2-\gamma; \mp z) . 
	\end{align}
Then 
	\begin{align}\label{3eq: Fourier B+-}
	A^{\pm}_{s, \vnu} (y) =   \cos (\pi \vnu) \Gamma (1-s, 2\vnu)   (iy)^{s-1} 	\Psi_{\pm} ( 1-s, 1+2\vnu; - i/y),
	\end{align}
for real $y \neq 0$ and $ 3/4 < \mathrm{Re} (s) < 1 $.
\end{lem}

\begin{lem}\label{lem: C(0)}
	We have
	\begin{align}\label{3eq: A(0)}
	A^{\pm}_{s, \vnu} (0)  =   \Gamma (1-s, 2\vnu)  \cdot  \left\{ \begin{aligned}
			& -   \cos (\pi (s+\vnu))   , & & \text{ if } +, \\
			& \,   \cos (\pi \vnu)  , & & \text{ if } -,
		\end{aligned} \right.    
	\end{align}  
	for $ 1/4 < \mathrm{Re} (s) < 1 $. 
\end{lem}

 \subsection{The Fourier Transform of Regularized Bessel Functions} 
 \label{sec: Fourier, reg}
 
 It is left to calculate the Fourier transform of $\varww (x )  R_{\vnu} (2 \sqrt{x})$ 
as defined  in \S \ref{sec: regularized Bessel}. More explicitly, let us write 
\begin{align}\label{3eq: W(y)}
	W_{s, \vnu} (y) =	\int_0^{\infty}   \varww (x )   R_{\vnu} (2 \sqrt{x})  \exp (-  i x y)  \frac { \nd      x} {x^{s+\vnu} }. 
\end{align}

By the Mellin inversion formula as in \eqref{3eq: Mellin inversion}, 
\begin{align*}
	\int_0^{\infty} \varww (x)  \exp (- i xy ) x^{  \vkappa -1 } \nd x = \frac 1 {2\pi i} 	 \int_{\sigma - i\infty}^{\sigma + i \infty}  
\widetilde{\varww} (\rho) \int_0^{\infty} x^{  \vkappa - \rho - 1} \exp ( - i xy ) \nd x \, \nd \rho  ,
\end{align*}
for $\max\{0, \mathrm{Re} (\vkappa) - 1 \} < \sigma <   \mathrm{Re} (\vkappa) $. 
By \cite[3.761 4, 9]{G-R},  for $a > 0$ we have
\begin{align*}
	\int _0^{\infty} x^{\mu - 1} \exp (\pm i a x  ) \nd x = \frac {\Gamma (\mu)} {a^{\mu}} \exp \lp \pm \frac {  \pi i \mu} 2 \rp ,  \qquad \text{($0 < \mathrm{Re} (\mu) < 1$)}. 
\end{align*}
Therefore 
\begin{align*}
\int_0^{\infty} \varww (x) \exp (- i xy )  x^{  \vkappa - 1} \nd x =   \frac 1 {2\pi i } 	 \int_{\sigma - i\infty}^{\sigma + i \infty}  
%
\widetilde{\varww} (\rho)   {\Gamma (\vkappa-\rho)}  (iy)^{\rho-\vkappa}  \nd \rho . 
\end{align*}
By shifting the integral contour to $\mathrm{Re} (\rho) = - 1 $ and collecting the residue at the pole of $ \widetilde{\varww} (\rho) $ at $\rho = 0$, we obtain
\begin{align}\label{3eq: Fourier w(x)}
	\int_0^{\infty}   \varww (x) \exp (- i xy )  x^{  \vkappa - 1} \nd x = \Gamma (\vkappa) (iy)^{-\vkappa} +  \frac 1 {2\pi i }  \!   \int_{-1 - i\infty}^{-1 + i \infty} 
%
\!  
\widetilde{\varww} (\rho)   {\Gamma (\vkappa-\rho)}  (iy)^{\rho-\vkappa}  \nd \rho . 
\end{align}

The next two lemmas now follow directly from  \eqref{3eq: A(x)}, \eqref{3eq: R(x)}, \eqref{3eq: Mellin}, \eqref{3eq: W(y)}, \eqref{3eq: Fourier w(x)}, and  the Euler reflection formula. 

\begin{lem}\label{lem: W(y)}
	Define 
	\begin{align}\label{3eq: Pi(z)}
	\Psi_0  (\valpha, \gamma ; z) = \frac {\Gamma (1-\gamma)} {\Gamma (1+ \valpha-\gamma )}  + \frac {\Gamma (\gamma-1)} {\Gamma (\valpha)} z^{1-\gamma}  . 
\end{align}
Then 
\begin{align}
\begin{aligned}
	W_{s, \vnu} (y)  =  A_{s, \vnu}^{0} (y) + \frac 1 {2\pi i} \int_{ -1 - i\infty}^{ -1 + i \infty} 
 \widetilde{\varww} (\rho ) A_{s+\rho, \vnu}^{0} (y) \nd \rho , 
\end{aligned}
\end{align}
for real $y \neq 0$ and  $ 0 < \mathrm{Re}(s) < 1$, where
\begin{align}\label{3eq: A0(y)}
A_{s, \vnu}^{0} (y) =	\cos (\pi \vnu) \Gamma (1-s, 2\vnu) (iy)^{s-1} 	\Psi_0 ( 1 - s, 1    +    2\vnu; - i/y) . 
\end{align}
\end{lem}

\begin{lem} \label{lem: W(0)} We have
	\begin{align}\label{3eq: W(0)}
	W_{s, \vnu} (0) = \cos (\pi \vnu) \big(  \widetilde{\varww} (1-s-2\vnu) \Gamma (2\vnu)   +  \widetilde{\varww} (1-s) \Gamma (-2\vnu)   \big),
	\end{align} 
for $ \mathrm{Re} (s) < 1 $. 
\end{lem}

In view of \eqref{3eq: M(x)},  \eqref{3eq: C(y)}, \eqref{3eq: N(y)}, and \eqref{3eq: W(y)}, 
\begin{align}\label{3eq: N=A-W}
	N^{\pm}_{s, \vnu} (y) = A^{\pm}_{s, \vnu} (y) - W_{s, \vnu} (y) , 
\end{align}
and we deduce the following formula for $N^{\pm}_{s, \vnu} (y)$ from Lemmas \ref{lem: C(y)} and \ref{lem: W(y)}. 

\begin{coro}\label{cor: Fourier}
Define the regularized Kummer functions 
\begin{align}\label{3eq: reg Phi}
	\breve{\Phi}  (\valpha, \gamma; z) = \Phi (\valpha, \gamma; z) - 1 , 
\end{align}
\begin{align}\label{3eq: reg Psi}
	\breve{\Psi}_{\pm} (\valpha, \gamma; z) \!   = \!   \frac {\Gamma (1-\gamma)} {\Gamma (1+ \valpha-\gamma )} \breve{\Phi} (\valpha, \gamma; \mp z) + \frac {\Gamma (\gamma-1)} {\Gamma (\valpha)} z^{1-\gamma} \breve{\Phi} (1+\valpha-\gamma, 2-\gamma; \mp z) . 
\end{align} 
Then we have the formula
\begin{align}\label{3eq: N(y)=A(y)}
	N^{\pm}_{s, \vnu} (y) =  \breve{A}^{\pm}_{s, \vnu} (y) -  \frac 1 {2\pi i} \int_{ -1 - i\infty}^{ -1 + i \infty} 
%
 \widetilde{\varww} (\rho ) A_{s+\rho, \vnu}^{0} (y) \nd \rho , 
\end{align}
for $y \neq 0$ and $ 3/4 < \mathrm{Re} (s) < 1 $,   where  
\begin{align}\label{3eq: reg A(y)}
	\breve{A}^{\pm}_{s, \vnu} (y)  =  \cos (\pi \vnu) \Gamma (1-s, 2\vnu) (iy)^{s-1} 	\breve{\Psi}_{\pm} ( 1-s, 1+2\vnu; - i/y) . 
\end{align} 
Consequently,  as   $|y| \ra \infty$ we have 
\begin{align}\label{3eq: N+-(x), decay}
	N^{\pm}_{s, \vnu} (y)  = O_{s}   \big(   {|y|^{ \mathrm{Re} (s) - 2}}   \big) . 
\end{align} 
\end{coro}

\delete{Note here that 
\begin{align*}
	 \breve{\Psi}_{\pm} (\valpha, \gamma; z) = {\Psi}_{\pm} (\valpha, \gamma; z) - \Psi_0 (\valpha, \gamma; z), \qquad  \breve{A}^{\pm}_{s, \vnu} (y) =  {C}^{\pm}_{s, \vnu} (y) - \breve{A}^{0}_{s, \vnu} (y). 
\end{align*}}


\section{Proof of Theorem \ref{thm: explicit}: the Explicit Formulae} \label{sec: proof}

\subsection{Setup} 

Let $\phi (t)  $ be as in Theorem \ref{thm: explicit}.  Let $s \in \BC$. Define 
\begin{align}
	\label{6eq: Cs(h)}
	\EuScript{C}_{\delta} (s; \phi) =    \sum_{f \in \SB_{\delta}} \omega_f (L (s+it_f, f) \phi  (t_f) + L (s-it_f, f) \phi (- t_f) ).  
\end{align}
Define
\begin{align}
	\label{6eq: Es(h)}
	\EuScript{E} (s; \phi) = \frac 2 {\pi} \int_{-\infty}^{\infty} \omega (t) {\zeta (s) \zeta (s+2it) } { }   \phi  (t) \nd t , \qquad \omega (t) = \frac {1 } {|\zeta (1 + 2it)|^2}, 
\end{align}
if $\mathrm{Re} (s) > 1$, and 
\begin{align}\label{6eq: Es(h), 2}
	\EuScript{E} (s; \phi) = \frac 2 {\pi} \int_{-\infty}^{\infty} \omega (t) {\zeta (s) \zeta (s+2it) }   \phi (t) \nd t  +  \frac { 2 \phi  ({(1-s)} / {2i} )  } { \zeta (2-s)   }      , 
\end{align}
if $ 0 < \mathrm{Re} (s) < 1$.  Note that \eqref{6eq: Es(h), 2} is indeed the analytic continuation of \eqref{6eq: Es(h)}. To see this, let us rewrite \eqref{6eq: Es(h)} as  
\begin{align*} 
		\EuScript{E} (s; \phi) =  \frac 1 {   \pi i}  \int_{-i \infty}^{i\infty}  
	\frac {\zeta (s) \zeta (s+ \vnu) } { \zeta (1+ \vnu) \zeta (1- \vnu) }   \phi  \Big( \frac {\vnu} {2i} \Big)  \nd \vnu , 
\end{align*}
for $\mathrm{Re}(s) > 1$,  and move the path of integral to $\mathrm{Re} (\vnu) = 1$, passing over the poles that correspond to the (non-trivial) zeros of $ \zeta (1- \vnu) $. This gives the analytic continuation of $	\EuScript{E} (s; \phi) $ to the strip $ 0 <\mathrm{Re} (s) < 1 $. Then, restricting to the strip, we shift  $\mathrm{Re} (\vnu) = 1$ down to $  \mathrm{Re} (\vnu) = 0$, now passing over the pole at $ \vnu = 1-s$ and those from the zeros of $ \zeta (1- \vnu) $ (again). In this way, we obtain the analytic continuation of $	\EuScript{E} (s; \phi) $ as in \eqref{6eq: Es(h), 2} by a calculation of the residues. Of course the definition \eqref{6eq: Es(h), 2} may be extended to the half-plane $ \mathrm{Re} (s) < 1$ by the principle of analytic continuation. 

For $\mathrm{Re}(s) > 1$, we have
\begin{align*}
	 L (s, f) = \sum_{n    =1}^{\infty}\frac{ \lambda_f (n)  }{n^s}, \qquad \zeta (s-it) \zeta (s+it) = \sum_{n    =1}^{\infty}\frac{ \tau_{it} (n)  }{n^s} .
\end{align*} 
Therefore
\begin{align}
	\label{6eq: Cs(h), 1}
	\EuScript{C}_{\delta} (s; \phi) = \sum_{n=1}^{\infty} {n^{-s}} \sum_{f \in \SB_{\delta}} \omega_f \lambda_{f} (n) \big(n^{-it_f}  \phi (t_f) + n^{it_f}  \phi (-t_f) \big), 
\end{align}
and similarly
\begin{align}
	\label{6eq: Es(h), 1}
	\EuScript{E} (s; \phi) = \sum_{n=1}^{\infty} {n^{-s}}  \cdot \frac 1 {\pi} \int_{-\infty}^{\infty} \omega (t)   \tau_{it} (n)  \big(n^{-it}  \phi (t) + n^{it}  \phi (-t) \big) \nd t . 
\end{align}

\subsection{Application of the Kuznetsov Formula} For $\mathrm{Re}(s) > 1$, if we choose $m = 1$ and $h (t) = n^{-it}  \phi (t) + n^{it}  \phi (-t) $ in the Kuznetsov trace formula in \eqref{2eq: Kuznetsov}, then we deduce from \eqref{6eq: Cs(h), 1} and \eqref{6eq: Es(h), 1} the identity: 
\begin{align}\label{6eq: C+E=D+O}
	  \EuScript{C}_{\delta} (s; \phi) +  (1-\delta) \EuScript{E} (s; \phi) = \EuScript{D} (\phi) + \EuScript{O}_+ (s; \phi) + (-1)^{\delta} \EuScript{O}_- (s; \phi), 
\end{align}
with
\begin{align}\label{6eq: D(h)}
	\EuScript{D} (\phi) = \frac 1 {\pi^2} \int_{-\infty}^{\infty} \phi (t)     \tanh (\pi t) t  \nd t , 
\end{align}
\begin{align}\label{6eq: O+(s;h)}
	\EuScript{O}_{\pm} (s; \phi)  = \frac {2} {\pi^2} \sum_{n=1}^{\infty} \sum_{c=1}^{\infty} \frac {S(1, \pm n; c)} {n^s c}      \int_{-\infty}^{\infty}  \frac {\phi (t)} {n^{it}}   B^{\pm}_{it} \bigg(\frac {4\pi \sqrt{n}} {c} \bigg)   {\tanh (\pi t)} t   \nd t . 
\end{align}
Next, we wish to apply the Poisson summation formula to the $n$-sum, so we need to move it inside the $t$-integral and the $c$-sum. To this end, let us assume for the moment
\begin{align*}
	\mathrm{Re} (s) > 
	\frac {13} {8}
\end{align*}
to ensure absolute convergence.  Since $ \phi (t) = O (\exp (- a |t|^2)) $ as in \eqref{1eq: decay of h},  according as $ c \Gt \!    \sqrt{n} $ or $ c \Lt \!    \sqrt{n} $, the convergence may be easily verified by the uniform bounds in \eqref{2eq: bound, 1} or \eqref{2eq: J(x), 2}, \eqref{2eq: K(x), 2},  along with the contour shift to $\mathrm{Re}(it) = 3/8$  in the former case as discussed below \eqref{2eq: H+-(x), 2}.    By moving the $n$-sum inside, we obtain
\begin{align}\label{5eq: O (s;phi)}
 \EuScript{O}_{\pm} (s; \phi)  = \frac {2} {\pi^2}  \sum_{c=1}^{\infty}  \frac 1 c   \int_{-\infty}^{\infty} \sum_{n=1}^{\infty} \frac {S(1, \pm n; c)} {n^{s+it}  }       B^{\pm}_{it} \bigg(\frac {4\pi \sqrt{n}} {c} \bigg)  {\phi (t)}    {\tanh (\pi t)} t   \nd t . 
\end{align}
As the $c$-sum is outermost, it is legitimate here to shift the contour to $\mathrm{Re}(it) = 0$ again. However, in view of \eqref{2eq: B(x), decay}, for the inner  $n$-sum  to be convergent, it only requires (in the case of $ \EuScript{O}_{+} (s; \phi)  $) that
\begin{align*}
	\mathrm{Re} (s) > \frac 3 4; 
\end{align*} 
indeed, in our later analysis  (before the analytic continuation), we shall mainly work in the region: 
\begin{align*}
	\frac 3 4 < \mathrm{Re}  (s) < 1. 
\end{align*}
Since the condition \eqref{4eq: conditions} would not hold, the Poisson summation formula in Lemma \ref{lem: Poisson} could not be applied directly to the $n$-sum. To this end, we propose to regularize the Bessel functions $ B^{\pm}_{it} (x) $ as in \S \ref{sec: regularized Bessel}.

\subsection{Regularization} According to  \eqref{3eq: M(x)}, we split $\EuScript{O}_{\pm} (s; \phi) $ as in \eqref{5eq: O (s;phi)} into 
\begin{align}\label{4eq: O = M+R}
	 \EuScript{O}_{\pm} (s; \phi) = \frac 2 {\pi^2} \sum_{c=1}^{\infty}   \frac{1} {c}  \int_{-\infty}^{\infty} ( \EuScript{M}_{\pm}  (s, it ; c)   + \EuScript{R}_{\pm} (s, it ; c) ) {\phi (t)}    {\tanh (\pi t)} t   \nd t  ,  
\end{align}
with  
\begin{align}\label{5eq: M(s,v;c;phi)}
	\EuScript{M}_{\pm} (s, \vnu ; c)   = \sum_{n=1}^{\infty} \frac {S(1, \pm n; c)} {n^{s+\vnu}  }       M^{\pm}_{\vnu} \bigg(\frac {4\pi \sqrt{n}} {c} \bigg) ,
\end{align}
\begin{align}\label{5eq: R+-}
	\EuScript{R}_{\pm}  (s, \vnu ; c) =     \sum_{n=1}^{\infty} \frac {S(1, \pm n; c)} {n^{s + \vnu}  }   \varww \bigg(\frac {4\pi^2  {n}} {c^2} \bigg)      R_{\vnu} \bigg(\frac {4\pi \sqrt{n}} {c} \bigg) . 
\end{align}

\subsection{Application of the Poisson Summation Formula} 
Let $\mathrm{Re}(\vnu) = 0$ but $\vnu \neq 0$.  After opening the Kloosterman sum $S (1, \pm n; c)$ in \eqref{5eq: M(s,v;c;phi)}, we obtain 
\begin{align*} 
	\EuScript{M}_{\pm} (s, \vnu ; c)   =   \sumx_{   a      (\mathrm{mod} \, c) } e \Big( \!    \pm \frac {\bar{a}} {c} \Big) \sum_{n=1}^{\infty} e \Big(  \frac {an} {c} \Big)  \frac {1} {n^{s+\vnu}  }       M^{\pm}_{\vnu} \bigg(\frac {4\pi \sqrt{n}} {c} \bigg) . 
\end{align*}
For the inner $n$-sum we would like to apply the Poisson summation formula in Lemma \ref{lem: Poisson}, with 
\begin{align*}
	 f (x) = \left\{ \begin{aligned}
	 &	\displaystyle  \frac {1} {x^{s+\vnu}  }       M^{\pm}_{\vnu} \bigg(\frac {4\pi \sqrt{x}} {c} \bigg), & & \text{ if } x > 0, \\
	 & 0, & & \text{ if } x \leqslant 0, 
	 \end{aligned} \right. 
\end{align*}
and, in view of \eqref{3eq: N(y)},
\begin{align*}
	\hat{f} (y) = \lp\frac {c} {2\pi} \rp^{2(1-s-\vnu)} N^{\pm}_{s, \vnu} \bigg(  \frac {c^2 y} {2\pi}\bigg)  . 
\end{align*} 
Thanks to the regularization, it follows from \eqref{2eq: B(x), decay}, \eqref{3eq: M+-(x), decay}, and \eqref{3eq: N+-(x), decay} that  the conditions in \eqref{4eq: conditions} are satisfied as long as $3/4 < \mathrm{Re}(s) < 1$. Thus by the Poisson summation formula \eqref{2eq: Poisson}, 
\begin{align*}
	\EuScript{M}_{\pm} (s, \vnu ; c)   =  \Big( \frac {c} {2\pi} \Big)^{2(1-s-\vnu)}  \sumx_{   a      (\mathrm{mod} \, c) } e \Big( \!    \pm \frac {\bar{a}} {c} \Big)  \sum_{n = - \infty}^{\infty}  N^{\pm}_{s, \vnu} \bigg(  \frac {c (c n-a) } {2\pi}\bigg) ,
\end{align*}
and by the change $ cn - a \ra  n $, we have
\begin{align}\label{4eq: M(c)}
	\EuScript{M}_{\pm} (s, \vnu ; c)   =   \Big( \frac {c} {2\pi} \Big) ^{2(1-s-\vnu)} \sum_{(n, c) = 1}   e \Big( \!    \mp \frac {\widebar{n}} {c} \Big)    N^{\pm}_{s, \vnu} \Big(  \frac {c n } {2\pi} \Big). 
\end{align}
Note that the only zero frequency arises in the case $c = 1$ and it equals 
\begin{align}\label{4eq: M0(c)}
\EuScript{M}^{0}_{\pm} (s, \vnu; 1) =	(2\pi)^{2(s+\vnu- 1)} { N^{\pm}_{s, \vnu} (0) } . 
\end{align}



\subsection{Application of the Lerch Functional Equation} By \eqref{3eq: A(x)}, \eqref{3eq: R(x)}, \eqref{3eq: Mellin inversion}, and the Euler reflection formula, we rewrite $\EuScript{R}_{\pm}  (s, \vnu ; c)  $ in \eqref{5eq: R+-} as follows:
\begin{align}\label{4eq: R(s,v;c)}
	\EuScript{R}_{\pm}  (s, \vnu ; c) =   \frac {1} {2\pi i}     \int_{\sigma - i\infty}^{\sigma + i \infty} 
%
\widetilde{\varww} (\rho) \cos  (\pi \vnu) (c/2\pi)^{2 \rho}  K_{\pm}  (s+\rho, 2 \vnu; c   ) \nd \rho  ,  
\end{align} 
for $\sigma > 1 - \mathrm{Re}(s) $, with  
\begin{align}\label{4eq: K=K+K}
	K_{\pm} (s, \vnu ; c ) =  K_{\pm} (s; c) (c/2\pi)^{  - \vnu} \Gamma (- \vnu) + K_{\pm} (s + \vnu ; c) (c/2\pi)^{  \vnu} \Gamma (\vnu)   , 
\end{align}
where 
\begin{align}\label{4eq: K(s;c), 0}
	K_{\pm} (s; c) =    \sum_{n=1}^{\infty}   \frac {S(1, \pm n; c)} {n^{s }  }, 
\end{align}
for $\mathrm{Re} (s) > 1$, while, by the definitions in \eqref{3eq: Kloosterman} and \eqref{2eq: Lerch}, 
\begin{align}\label{4eq: K(s;c)}
	K_{\pm} (s; c) = \sumx_{   a      (\mathrm{mod} \, c) } e \Big(   \frac {a \pm \widebar{a}} {c} \Big) \zeta \Big(s, 1, \frac a c \Big), 
\end{align}
 of which its analytic continuation is a direct consequence. 
Note that $K_{\pm} (s; c)$ is entire except for $c = 1$, in which case 
\begin{align}\label{4eq: K(s;1)}
	K_{\pm} (s; 1) =  \zeta (s). 
\end{align}
For $c > 1$ the functional equation \eqref{2eq: Lerch, FE} yields 
\begin{align} 	\label{4eq: K(s;c), 2}
K_{\pm} (1-s; c) =	 \frac {\Gamma (s)}    {(2\pi)^{s}} \sumx_{   a      (\mathrm{mod} \, c) } \!   e \Big(  \!   \pm \frac {   \widebar{a}} {c} \Big)  \bigg\{ e \Big(  \frac {s} {4}   \Big) \zeta \Big(s, \frac {a} {c} \Big)   + e \Big( \!   - \frac {s} {4}    \Big) \zeta \Big(s,   \frac {c-a} {c} \Big)   \bigg \} . 
\end{align} 
By  applying the functional equations in \eqref{2eq: Lerch, FE, 2} and \eqref{2eq: Riemann, FE, 2} for $ \zeta (s, 1, a/c)$ and $\zeta (s)$ as in \eqref{4eq: K(s;c)} and \eqref{4eq: K(s;1)}, 
 for $\mathrm{Re} (s) > 1$ we obtain  
\begin{align*} 
		K_{\pm} (1-s; c) = \Gamma (s)  \sumx_{   a      (\mathrm{mod} \, c) } e \Big( \!    \pm \frac {\widebar{a}} {c} \Big) \sideset{}{^{_{\,  {\scriptstyle \prime}}}}\sum_{n = - \infty}^{\infty} \frac 1 {(2\pi i (n-a/c))^s},  
\end{align*}
($a = 0$ if $c = 1$) and, by the change  $ cn - a \ra  n $ again, we arrive at 
\begin{align}\label{4eq: K(s;c), 2.1}
	K_{\pm} (1-s; c) = \Gamma (s) \Big( \frac {c} {2\pi} \Big)^{s}  \sideset{}{^{_{\,  {\scriptstyle \prime}}}} \sum_{(n, c) = 1}   e \Big( \!    \mp \frac {\widebar{n}} {c} \Big) \frac 1 {(in)^s}; 
\end{align}
here and henceforth  the prime on the $n$-sums means $ n \neq 0$ in the case $c = 1$. 
Moreover,  it follows from \eqref{3eq: Lerch-Dirichlet} and \eqref{4eq: K(s;c), 2} that 
\begin{align}	\label{4eq: K(s;c), 3}
	K_{\pm} (1-s; c) =    {\Gamma (s)}   \Big( \frac {c} {2\pi} \Big)^{s}    \frac {1} {\varphi (c)}  \sum_{   \vchi      (\mathrm{mod} \, c) }  \tau (\vchi) \epsilon_{\pm} (s, \vchi)   L  (s, \vchi  ) , 
\end{align} 
where as usual $\tau (\vchi)$ is the Gauss sum
\begin{align}\label{3eq: defn Gauss}
	\tau (\vchi) = \sumx_{   a      (\mathrm{mod} \, c) } \vchi (a) e \Big(   \frac {  a } {c} \Big),
\end{align}
and 
\begin{align}\label{4eq: epsilon (chi)}
\epsilon_{\pm} (s, \vchi)  
=     \left\{ \begin{aligned}
	&2 \cos (\pi s/2), & &  \text{if } \vchi (-1) = 1, \\
	& \!   \pm 2 i \sin (\pi s/2), & & \text{if } \vchi (-1) = - 1. 
\end{aligned} \right.  
\end{align}
Note that \eqref{4eq: K(s;c), 3} is still valid for $c = 1$ as it is reduced to the Riemann functional equation \eqref{2eq: Reimann, FE}. From \eqref{4eq: K(s;c), 3} we  deduce a crude estimate for $  K_{\pm} (s ; c) $. 

\begin{lem} \label{lem: bound for K} We have
	\begin{align*}
		\frac {K_{\pm} (s ; c)} {c}  = O_{\vepsilon, \mathrm{Re}(s)}  \big( (c|s|)^{ (1 - \mathrm{Re}(s) ) / 2 + \vepsilon }  \big),
	\end{align*}
	 for $ 0 < \mathrm{Re}(s) < 1 $. 
\end{lem}

\begin{proof}
	This lemma follows from the Stirling formula and the bounds $$\tau (\vchi) = O (\sqrt{c}), \qquad L (s, \vchi ) = O_{\vepsilon, \mathrm{Re}(s)} \big((c |s|)^{  (1-\mathrm{Re}(s))/ 2 + \vepsilon } \big) , $$ for $0 < \mathrm{Re}(s) < 1$; the latter is the trivial convexity bound (see \cite[(5.20)]{IK}). For the case of non-primitive $\vchi$ we refer the reader to  \cite[\S \S 5, 9]{Davenport-Mult-NT}. 
\end{proof}

Next,  we shift the integral contour in \eqref{4eq: R(s,v;c)} to $\mathrm{Re}(\rho) = - 1$.    
It follows the identities in \eqref{3eq: Gamma}, \eqref{3eq: Pi(z)},  \eqref{3eq: A0(y)}, \eqref{4eq: K=K+K}, and \eqref{4eq: K(s;c), 2.1}   that the integral is transformed into
\begin{align}\label{4eq: W(c)}
	\EuScript{W}_{\pm} (s, \vnu; c) =  \Big( \frac {c} {2\pi} \Big)^{2(1-s-\vnu)} \sideset{}{^{_{\,  {\scriptstyle \prime}}}} \sum_{(n, c) = 1}   e \Big( \!    \mp \frac {\widebar{n}} {c} \Big)    \frac 1 {2\pi i}  
	 \int_{-1 - i\infty}^{-1 + i \infty} 
	  \widetilde{\varww} (\rho)  A_{s+\rho, \vnu}^{0} \Big(  \frac {c n } {2\pi} \Big) \nd \rho. 
\end{align}
As for the residues, the simple pole of $  \widetilde{\varww} (\rho) $ at $\rho = 0$ contributes
\begin{align}\label{4eq: K(s,v;c)}
 \EuScript{K}_{\pm} (s, \vnu; c) = \cos (\pi \vnu) K_{\pm} (s, 2\vnu; c), 
\end{align}
while, in the case $c = 1$,  the simple poles of $ \zeta (s+\rho + 2\vnu) $ and  $ \zeta (s+\rho) $ contribute  
\begin{align}\label{4eq: W0(1)}
\EuScript{W}_{ 0} (s, \vnu) =  (2\pi)^{2(s+\vnu- 1)}	W_{s, \vnu} (0) ,
\end{align}
by the expression of $ W_{s, \vnu} (0) $ as in \eqref{3eq: W(0)}. We conclude that 
\begin{align}\label{4eq: R}
	\EuScript{R}_{\pm}  (s, \vnu ; c) =   \EuScript{W}_{\pm} (s, \vnu; c) + \EuScript{K}_{\pm} (s, \vnu; c)   + \delta (c, 1) \EuScript{W}_{ 0} (s, \vnu). 
\end{align} 

\subsection{Cancellation} 
By \eqref{3eq: N=A-W}, \eqref{3eq: N(y)=A(y)}, \eqref{4eq: M(c)}, \eqref{4eq: M0(c)}, \eqref{4eq: W(c)}, \eqref{4eq: W0(1)},  after cancellation, we infer that 
\begin{align*}
	 \EuScript{M}_{\pm} (s, \vnu ; c) + \EuScript{W}_{\pm}   (s, \vnu ; c) =  \Big( \frac {c} {2\pi} \Big) ^{2(1-s-\vnu)} \sum_{(n, c) = 1}   e \Big( \!    \mp \frac {\widebar{n}} {c} \Big)    \breve{A}^{\pm}_{s, \vnu} \Big(  \frac {c n } {2\pi} \Big), 
\end{align*}
if $c > 1$, and 
\begin{align*}
	 \EuScript{M}_{\pm}^{0}  (s, \vnu ; 1) + \EuScript{W}_{0}   (s, \vnu)   =  (2\pi)^{2(s+\vnu- 1)} { A^{\pm}_{s, \vnu} (0) }, 
\end{align*}
\begin{align*}
	 \EuScript{M}_{\pm}(s, \vnu ; 1)  + \EuScript{W}_{\pm} (s, \vnu ; 1)  + \EuScript{W}_{0}  (s, \vnu ) =  (2\pi)^{2(s+\vnu- 1)}  \bigg\{ { A^{\pm}_{s, \vnu} (0) } +  \sum_{n\neq 0} \breve{A}^{\pm}_{s, \vnu} \Big(  \frac {  n } {2\pi} \Big)  \bigg \}  , 
\end{align*}
where  $\breve{A}^{\pm}_{s, \vnu} (y)$ and   $ A^{\pm}_{s, \vnu} (0) $ are given explicitly in  \S \S \ref{sec: Fourier} and \ref{sec: Fourier, reg} (see particularly Corollary \ref{cor: Fourier} and Lemma  \ref{lem: C(0)}). 
At any rate, all the $\widetilde{\varww}$ get  canceled out as expected.  For uniformity, we write
\begin{align}\label{4eq: M+W}
	 \EuScript{M}_{\pm} (s, \vnu ; c) + \EuScript{W}_{\pm}   (s, \vnu ; c) + \delta (c, 1)  \EuScript{W}_{0}  (s, \vnu ) = \breve{ \EuScript{A}}_{\pm} (s, \vnu ; c) +  \delta (c, 1) \EuScript{A}_{\pm}^0 (s, \vnu ), 
\end{align}
where 
\begin{align}\label{4eq: A}
\breve{ \EuScript{A}}_{\pm} (s, \vnu ; c) =	\Big( \frac {c} {2\pi} \Big) ^{2(1-s-\vnu)} \sideset{}{^{_{\,  {\scriptstyle \prime}}}} \sum_{(n, c) = 1}   e \Big( \!    \mp \frac {\widebar{n}} {c} \Big)    \breve{A}^{\pm}_{s, \vnu} \Big(  \frac {c n } {2\pi} \Big), 
\end{align}
\begin{align}\label{4eq: A0}
	\EuScript{A}_{\pm}^0 (s, \vnu ) = (2\pi)^{2(s+\vnu- 1)}    A^{\pm}_{s, \vnu} (0). 
\end{align}
We conclude from \eqref{4eq: R} and \eqref{4eq: M+W} that
\begin{align}\label{4eq: R+M}
	\EuScript{M}_{\pm} (s, \vnu ; c) + \EuScript{R}_{\pm}   (s, \vnu ; c)  = \EuScript{K}_{\pm} (s, \vnu ; c)  +  \breve{ \EuScript{A}}_{\pm} (s, \vnu ; c) +  \delta (c, 1) \EuScript{A}_{\pm}^0 (s, \vnu ). 
\end{align}

\subsection{The Analytic Continuation}\label{sec: anal cont}
Note that each term on the right of \eqref{4eq: R+M} admits analytic continuation to the critical strip  $$ 0 < \mathrm{Re} (s) < 1 . $$ 
However, we still need to take into account of the $c$-sum and $t$-integral in \eqref{4eq: O = M+R}. Let us first consider the latter two $\EuScript{A}$-terms as the first $\EuScript{K}$-term will require more work on   the (absolute) convergence. Keep in mind that  $\phi (t)$ decays like $\exp (- a |t|^2)$, so the $t$-integral is easy to handle.

For   brevity, after the analytic continuation, we shall usually drop $s$ from the notation if the formula is evaluated at $s = 1/2$. 


\vspace{5pt}

\subsubsection{Analysis for the $ \EuScript{A}$-terms} 
It may be verified directly that $\breve{ \EuScript{A}}_{\pm} (s, \vnu ; c)$ and $\EuScript{A}_{\pm}^0 (s,  \vnu)$ ($\vnu = it$)  contribute analytic functions of $s$ on the critical strip due to the absolute (uniform) convergence. 
For simplicity, let us directly specialize the formulae at the center $s = 1/2$, obtaining 
\begin{align}\label{4eq: A(phi)}
	\breve{\EuScript{A}}_{\pm} (\phi) = \frac { {2}} {  \pi^2 i}  \int_{-\infty}^{\infty}  \breve{ {A}}_{\pm} (2it)  {\phi (t)}    t   \nd t  ,  \ \ \ 
	 {\EuScript{A}}_{\pm}^{0} (\phi) = \frac {2} {  \pi^2  } \int_{-\infty}^{\infty}  \gamma_{\pm}^{0} (  2  i t ) {\phi (t)}  \tanh  (\pi t)     t  \nd t  , 
\end{align}
where 
\begin{align} 
\breve{ {A}}_{\pm} (\vnu) = \frac {\sin (\pi \vnu/2)} {\sqrt{2}  }  \Gamma \bigg( \frac 1 2 - \vnu \bigg)	
\sum_{c> 0} \sumx_{n\neq 0} 
\frac { e  (   \mp   {\widebar{n}} / {c}  ) } {(c/2\pi)^{\vnu} \sqrt{i c n} }	 \breve{\Psi}_{\pm} \bigg( \frac 1 2  , 1+ \vnu  ; \frac {2\pi} {i c n}  \bigg) ,
\end{align} 
the superscript $\star$  indicates $(n, c) = 1$, and 
\begin{align}\label{4eq: A+-0}
	\gamma_{\pm}^{0} (\vnu) =  \frac { (2\pi)^{ \vnu - \frac 1 2 }  } {  \sqrt{2}  }  \Gamma \bigg( \frac 1 2 -  \vnu \bigg)   \cdot  \left\{ \begin{aligned}
		&     \sin (\pi  \vnu/2 )   , & & \text{ if } +, \\
		&    \cos (\pi \vnu/2 )  , & & \text{ if\,} - . 
	\end{aligned} \right.    
\end{align}  
According to \eqref{3eq: reg Psi}, we split  $\breve{ {A}}_{\pm} (\vnu) $ into the sum of $ \breve{ {A}}_{\pm}^{1} (\vnu) $ and $ \breve{ {A}}_{\pm}^{\flat} (\vnu) $, defined by
\begin{align}\label{4eq: A flat}
\breve{ {A}}_{\pm}^{1} (\vnu) =  \frac {\sin (\pi \vnu/2)  } {\sqrt{2}  } \Gamma (-\vnu)   \sum_{c> 0} \sumx_{n\neq 0}  
\frac { e  (   \mp   {\widebar{n}} / {c}  ) } {(c/2\pi)^{\vnu} \sqrt{i c n} }	   \breve{\Phi}  \bigg( \frac 1 2  , 1+ \vnu  ; \pm \frac {2\pi i} { c n}  \bigg) , 
\end{align}
\begin{align}
	\breve{ {A}}_{\pm}^{\flat} (\vnu) = \frac {\sin (\pi \vnu/2)  } {\sqrt{2 \pi}  }  \Gamma (\vnu) \Gamma \bigg( \frac 1 2 - \vnu \bigg) \sum_{c> 0} \sumx_{n\neq 0} \!   
	 \frac { e  (   \mp   {\widebar{n}} / {c}  ) } {(i n)^{- \vnu} \sqrt{i c n} }	 \breve{\Phi}  \bigg( \frac 1 2 - \vnu , 1 - \vnu  ; \pm \frac {2\pi i} { c n}  \bigg) . 
\end{align}
Now we invoke the Kummer transformation formula from \cite[6.3 (7)]{Erdelyi-HTF-1}
\begin{align*}
	\Phi (\valpha, \gamma; z) = \exp (z) \Phi (\gamma-\valpha, \gamma; - z);
\end{align*}
 in view of \eqref{3eq: reg Phi}, we have its regularized variant
\begin{align}\label{3eq: Kummer transform}
	\breve{\Phi} (\valpha, \gamma; z) = \exp (z) \breve{\Phi} (\gamma-\valpha, \gamma; - z) + \exp (z) - 1. 
\end{align}
By 
\eqref{3eq: Kummer transform}, we further split $\breve{ {A}}_{\pm}^{\flat} (\vnu) $ into the sum of $\breve{ {A}}_{\pm}^{\natural} (\vnu) $ and $\breve{ {K}}_{\pm}  (\vnu) $, defined by 
\begin{align}\label{4eq: A+-natural}
	\breve{ {A}}_{\pm}^{\natural} (\vnu) =  \frac {\sin (\pi \vnu/2)  } {\sqrt{2 \pi} }  \Gamma (\vnu) \Gamma \bigg( \frac 1 2 - \vnu \bigg) \sum_{c> 0} \sumx_{n\neq 0}  
	 \frac { e  (   \pm   {\widebar{c}} / {n}  ) } {(i n)^{-\vnu} \sqrt{i c n} }	  \breve{\Phi}  \bigg( \frac 1 2  , 1 - \vnu  ; \mp \frac {2\pi i} { c n}  \bigg) ,
\end{align}
\begin{align}\label{4eq: breve K}
\breve{ {K}}_{\pm}  (\vnu) = \frac {\sin (\pi \vnu/2)  } {\sqrt{2 \pi} }  \Gamma (\vnu) \Gamma \bigg( \frac 1 2 - \vnu \bigg)  \sum_{c> 0} \sumx_{n\neq 0}  
 \frac { e  (   \pm   {\widebar{c}} / {n}  ) } {(i n)^{-\vnu} \sqrt{i c n} }	  \bigg\{ 1 - e \bigg( \!    \mp \frac {1} {c n} \bigg)  \bigg\} ;
\end{align}
here we have also applied the reciprocity formula 
\begin{align}\label{4eq: reciprocity}
	e \Big( \!    \mp \frac {\widebar{n}} {c} \Big) = e \Big( \!    \pm \frac {\widebar{c}} {n} \Big) e \bigg( \!    \mp \frac {1} {c n} \bigg) . 
\end{align} 
Moreover, it is better to rewrite 
\begin{align}\label{4eq: A+-natural, 2}
	\breve{ {A}}_{\pm}^{\natural} (\vnu) =  \frac { \sqrt{\pi} } {2 \sqrt{2  }  \cos (\pi \vnu/2)}  \frac { \Gamma (1/2-\vnu) } {\Gamma (1-\vnu)}  \sum_{c> 0} \sumx_{n\neq 0}  
	\frac { e  (   \pm   {\widebar{c}} / {n}  ) } {(i n)^{-\vnu} \sqrt{i c n} }	  \breve{\Phi}  \bigg( \frac 1 2  , 1 - \vnu  ; \mp \frac {2\pi i} { c n}  \bigg) .
\end{align}
Similar to \eqref{4eq: A(phi)}, let us denote the contributions of $ \breve{ {A}}_{\pm}^{1} (\vnu) $ and $\breve{ {A}}_{\pm}^{\natural} (\vnu)$ as follows:
\begin{align}\label{4eq: A(phi), 2}
	\breve{\EuScript{A}}_{\pm}^{1}  (\phi) = \frac { {2}} {  \pi^2 i}  \int_{-\infty}^{\infty}  \breve{ {A}}_{\pm}^{1}  (2it)  {\phi (t)}    t   \nd t  ,  \quad 
\breve{\EuScript{A}}_{\pm}^{\natural} (\phi) = \frac { {2}} {  \pi^2 i}  \int_{-\infty}^{\infty}  \breve{ {A}}_{\pm}^{\natural} (2it)  {\phi (t)}    t   \nd t, 
\end{align}
while that of $ \breve{ {K}}_{\pm}  (\vnu)  $ will eventually be canceled. 

The reader should have observed the similarity between the $(c, n)$-sums in \eqref{4eq: A flat} and \eqref{4eq: A+-natural, 2} with respect to the symmetry $ c \xleftrightarrow{\ \  }   n$.

\vspace{5pt}

\subsubsection{Analysis for the $ \EuScript{K}$-terms} 
It is left to consider the analytic continuation of the contribution from $  \EuScript{K}_{\pm} (s, \vnu ; c) $ (see \eqref{4eq: K(s,v;c)}), which, in view of \eqref{4eq: K=K+K}, is split  into the sum of $  \EuScript{K}_{\pm}^{1} (s;  \phi) $ and $ \EuScript{K}_{\pm}^{\flat} (s;  \phi)$, defined by
\begin{align}
	\EuScript{K}_{\pm}^{1} (s;  \phi) = \frac 2 {\pi^2} \sum_{c=1}^{\infty}   \frac {K_{\pm} (s; c)} {c}      \int_{-\infty}^{\infty}   {(2\pi / c)^{2it}}  {\phi (t)}   \Gamma (-2it) {\sinh (\pi t)} t   \nd t , 
\end{align}
\begin{align}\label{4eq: K(s; phi), 2}
	\EuScript{K}_{\pm}^{\flat} (s;  \phi) = \frac 2 {\pi^2} \sum_{c=1}^{\infty}     \int_{-\infty}^{\infty}  \frac {K_{\pm} (s+2it; c)} {c} (c/2\pi)^{2it}   {\phi (t)}   \Gamma (2it) {\sinh (\pi t)} t   \nd t . 
\end{align}

In view of the (trivial) bound for ${K_{\pm} (s; c)} / {c}$ in Lemma \ref{lem: bound for K}, if  the integral contour is shifted to $\mathrm{Re} (it) = 1$ (it is safe to pass through $t = -i/2$ since by our assumption $\phi (-i/2) = 0$), then the formula of $	\EuScript{K}_{\pm}^{1} (s;  \phi)$ becomes absolutely convergent, and hence yields its analytic continuation on the critical strip.  
Consequently, if we insert the expression of $ K_{\pm} (s; c) $ in  \eqref{4eq: K(s;c), 3} and  specialize the formula to $s = 1/2$, then we obtain 
\begin{align}\label{5eq: K1+-}
	\EuScript{K}_{\pm}^{1} ( \phi) = \sum_{c=1}^{\infty}    {I (c  ; \phi)  }   \cdot \frac 1 {\sqrt{c} \varphi (c)}  \sum_{   \vchi      (\mathrm{mod} \, c) }  \tau (\vchi) \epsilon_{\pm} ( \vchi)   L   (   1 / 2 , \vchi   ) , 
\end{align}
with 
\begin{align}\label{5eq: I(c, phi)}
	I ( c ; \phi) = -\frac 2 {\pi^2  } 
	\int_{ i - \infty}^{i+\infty}  {c^{  2it}}  {\phi (- t)}  \gamma_{1}  (2it) t   \nd t,  \qquad  \gamma_{1}  (\vnu) =	i { (2\pi)^{ - \vnu }  }    \Gamma  ( \vnu  )      \sin  (\pi \vnu/2 ), 
\end{align}
\begin{align}\label{5eq: epsilon +-(chi)}
	\epsilon_{\pm} (\vchi) = \left\{ \begin{aligned}
		& 1, & & \text{ if } \vchi (-1) = 1, \\
		& \pm i, & & \text{ if } \vchi (-1) = - 1 .
	\end{aligned} \right. 
\end{align}

However, in the case of $	\EuScript{K}_{\pm}^{\flat} (s;  \phi)$ for $0 < \mathrm{Re}(s) < 1$, the contour shift fails to yield absolute convergence of the $c$-sum since one is left with $(c/2\pi)^{1 - s }$ after the cancellation between  $  (c/2\pi)^{2it} $ and  the factor $ (c/2\pi)^{1 - s - 2 it}  $ in $K_{\pm} (s+2it; c)$ (see \eqref{4eq: K(s;c), 3}). Our idea to address this issue is to use the symmetry  $ c \xleftrightarrow{\ \  }   n$ and interchange the roles of $c$ and $n$. To this end, let us shift the contour to $\mathrm{Re} (it) = - 1$ (by the condition $  \phi (i/2) = 0 $) and temporarily assume $$1 < \mathrm{Re} (s) < 2 $$
so that  \eqref{4eq: K(s;c), 2.1} may be applied to transform \eqref{4eq: K(s; phi), 2} into 
\begin{align*}
	\EuScript{K}_{\pm}^{\flat} (s;  \phi) = \frac 2 {\pi^2 i}    \int_{i-\infty}^{i+\infty} K_{\pm}^{\flat} (s, 2it)  {\phi (t)} t   \nd t ,
\end{align*}
with 
\begin{align*}
K_{\pm}^{\flat} (s, \vnu) = \frac {\sin( \pi \vnu/2)} {2\pi   }  \Gamma (\vnu)  \Gamma (1-s-\vnu)  \sum_{c> 0} \sumx_{n\neq 0}  
\frac {e (\mp \widebar{n}/ c)} {   (c/2\pi)^{s} (in)^{1-s-\vnu} }. 
\end{align*}
Next, we split the reciprocity formula in \eqref{4eq: reciprocity},
\begin{align*}
	e \Big( \!    \mp \frac {\widebar{n}} {c} \Big)  = e \Big( \!    \pm \frac {\widebar{c}} {n} \Big) - e \Big( \!    \pm \frac {\widebar{c}} {n} \Big) \bigg\{ 1 - e \bigg( \!    \mp \frac {1} {c n} \bigg)  \bigg \} ,
\end{align*} so that   
\begin{align*}
	K_{\pm}^{\flat} (s, \vnu) = K_{\pm}^{\natural} (s, \vnu) - \breve{K}_{\pm} (s, \vnu) , 
\end{align*}
where 
\begin{align*}
	K_{\pm}^{\natural} (s, \vnu) = \frac {\sin( \pi \vnu/2)} {2\pi   }   \Gamma (\vnu)  \Gamma (1-s-\vnu) \sum_{c> 0} \sumx_{n\neq 0}  
	 \frac {e (\pm \widebar{c}/ n)} {   (c/2\pi)^{s} (in)^{1-s-\vnu} }, 
\end{align*}
\begin{align*}
	\breve{K}_{\pm} (s, \vnu) = \frac {\sin( \pi \vnu/2)} {2\pi   }   \Gamma (\vnu)  \Gamma (1-s-\vnu) \sum_{c> 0} \sumx_{n\neq 0}  
	\frac {e (\pm \widebar{c}/ n)} {   (c/2\pi)^{s} (in)^{1-s-\vnu} } \bigg\{ 1 - e \bigg( \!    \mp \frac {1} {c n} \bigg)  \bigg \}. 
\end{align*}
Note that $ \breve{K}_{\pm} (1/2, \vnu) $ is exactly $ \breve{K}_{\pm} ( \vnu)  $ as defined in \eqref{4eq: breve K},  so  their contributions (of course, contour shift is needed) cancel each other! 
Moreover, it follows from 
\begin{align*}
	e \Big(\frac {a} {n} \Big) = \frac 1 {\varphi (n)}  \sum_{   \vchi      (\mathrm{mod} \, n) } \widebar{\vchi} (a) \tau (\vchi) , \qquad \text{($(a, n) = 1$),}
\end{align*}
that 
\begin{align}
\begin{aligned}
		K_{\pm}^{\natural} (s, \vnu) =   \frac {\sin( \pi \vnu/2)} {2\pi  }   \Gamma (\vnu)   (2\pi)^s  \sum_{n=1}^{\infty} n^{ \vnu}   {K_{\pm} (s, 1-s-\vnu ; n)}    , 
\end{aligned}
\end{align}
with
\begin{align}
K_{\pm} (s, w ; n) =	\Gamma ( w) \frac 1 {n^{1-s} \varphi (n)}  \sum_{   \vchi      (\mathrm{mod} \, n) } \tau (\vchi) \epsilon_{\pm} (-w,  \vchi) L(s, \vchi) ,
\end{align}
where $ \tau (\vchi) $ and $  \epsilon_{\pm} (s, \vchi) $ are defined as in \eqref{3eq: defn Gauss} and \eqref{4eq: epsilon (chi)}. Clearly $ K_{\pm} (s, w ; n)  $ admits analytic continuation on the critical strip   $ 0 < \mathrm{Re}(s) < 1 $, away from the poles of $\Gamma (w)$, and it has the following (crude) uniform bound similar to Lemma \ref{lem: bound for K}. 

\begin{lem} \label{lem: bound for K(s,w;n)} We have
	\begin{align*}
		K_{\pm} (s, w ; n)   = O_{\vepsilon, \mathrm{Re}(w), \mathrm{Re}(s)}  \big( |w|^{\mathrm{Re}(w)- 1/2} |s|^{ (1 - \mathrm{Re}(s) ) / 2 + \vepsilon }  n^{\mathrm{Re}(s)/2 + \vepsilon}   \big),
	\end{align*}
	for $ 0 < \mathrm{Re}(s) < 1 < \mathrm{Re} (w) $. 
\end{lem}

The analytic continuation of $ K_{\pm}^{\natural} (s, \vnu) $ on the domain
\begin{align*}
	0 < \mathrm{Re} {(s)} < 1, \qquad \mathrm{Re} (\vnu) < - 1 - \frac {\mathrm{Re} (s)} 2  , 
\end{align*}
may be readily deduced from Lemma \ref{lem: bound for K(s,w;n)}. Thus its contribution admits analytic  continuation to $0 < \mathrm{Re} {(s)} < 1$, and at $s = 1/2$ it reads
\begin{align}\label{5eq: K natural}
		\EuScript{K}_{\pm}^{\natural} (  \phi) = \sum_{n=1}^{\infty}  \frac 1 {\sqrt{n} \varphi (n)}  \sum_{   \vchi      (\mathrm{mod} \, n) }  \tau (\vchi)  I_{\pm} (n ; \phi; \vchi)   L   (   1 / 2 , \vchi   ), 
\end{align}
\begin{align}
	I_{\pm} (n ; \phi; \vchi) = \frac 2 { \pi^2 i}  
	\int_{ i - \infty}^{i+\infty}  {n^{  2it}}  {\phi ( t)} 	\gamma_{\pm}^{\natural} (2it, \vchi)  t   \nd t , 
\end{align}
\begin{align}\label{5eq: gamma (chi)}
	\gamma_{\pm}^{\natural}  (\vnu , \vchi) =   \frac {\sqrt{\pi}} {2  }  \frac { \Gamma (1/2-\vnu) } {\Gamma (1-\vnu)} \cdot  \left\{ \begin{aligned}
		&  (\tan (\pi \vnu/2) + 1), & &  \text{if } \vchi (-1) = 1, \\
		& \pm i (\tan (\pi \vnu/2) - 1), & & \text{if } \vchi (-1) = - 1, 
	\end{aligned} \right.  
\end{align}
where we have used the Euler reflection formula and simple trigonometric identities. 

\subsection{Conclusion} 

Given the arguments above, let us conclude that 
\begin{align}
	 \EuScript{O}_{\pm} (  \phi) =  {\EuScript{A}}_{\pm}^{0} (\phi) + \breve{\EuScript{A}}_{\pm}^{1}  (\phi) + \breve{\EuScript{A}}_{\pm}^{\natural}  (\phi) + 	\EuScript{K}_{\pm}^{1} ( \phi) + 	\EuScript{K}_{\pm}^{\natural} (  \phi); 
\end{align}
the definitions of the terms on the right may be found in \eqref{4eq: A(phi)}, \eqref{4eq: A+-0}, \eqref{4eq: A flat}, \eqref{4eq: A+-natural, 2}, \eqref{4eq: A(phi), 2}, \eqref{5eq: K1+-}--\eqref{5eq: epsilon +-(chi)}, and \eqref{5eq: K natural}--\eqref{5eq: gamma (chi)}. 
The proof of Theorem   \ref{thm: explicit} is completed by some rearrangements and simplifications of the terms in $  \EuScript{O}_{+} (  \phi) + (-1)^{\delta} \EuScript{O}_{-} (  \phi)   $.

\delete{ Similar to the argument below \eqref{6eq: O+(s;h)}, let us assume temporarily $$\mathrm{Re}(s) >   \frac {7}  {4} , $$ 
insert  
\begin{align*}
	K_{\pm} (s+2it; c) = \sum_{n=1}^{\infty}   \frac {S(1, \pm n; c)} {n^{s + 2it }  } , 
\end{align*}  (see  \eqref{4eq: K(s;c), 0}),  and shift the contour to $\mathrm{Re} (it) = - 3/8$ so that it is free to rearrange the order of the $c$-, $n$-sums and the $t$-integral. It follows that 
\begin{align}
	\EuScript{K}_{\pm}^{\flat} (s;  \phi) = \frac 2 {\pi^2}     \int_{\langle   3/8 \rangle } L_{\pm} (s; 2it)     {\phi (t)}   \Gamma (2it) {\sinh (\pi t)} t   \nd t , 
\end{align}
with 
\begin{align}
	L_{\pm} (s; \vnu) = \frac 1 {(2\pi)^{\vnu}} \sum_{c=1}^{\infty} \sum_{n=1}^{\infty} \frac {S (1, \pm n; c) } { c^{1-\vnu} n^{s+\vnu}  }. 
\end{align}
}

 	\begin{appendices}

 		\section{Explicit Formulae for the Twisted Mean Values}

 		Let us retain the notation in \S \ref{sec: explicit formulae}.  
 		Consider the twisted mean value
 		\begin{align} 
 			\EuScript{C}_{\delta} ( m; \phi) =    \sum_{f \in \SB_{\delta}} \omega_f \lambda_{f} (m) \big(   L (s_f, f) \phi  (t_f) +   L (\widebar{s}_f, f) \phi (- t_f) \big).  
 		\end{align} 
 	Further we extend Definitions \ref{defn: A(s)} and \ref{defn: L(c)} as follows. 
 	\begin{defn}\label{defn: A(m;s)}
 		For $\mathrm{Re} (s) = 0$, define the double series 
 		\begin{align}\label{app: defn A}
 			\breve{A}_{\delta}  (m; s) = \sum_{\pm} (\pm)^{\delta}	 \sum_{c=1}^{\infty}  \bigg( \frac c {\sqrt{m}}  \bigg)^{ s}    \mathop{ \sum_{n=1}^{\infty}  }_{(n, c) = 1}     \frac { e  ( \pm     {m \widebar{n} } / {c}  ) } { \sqrt{ c n} }	  \breve{\Theta}_{s}   \bigg( \!   \mp \frac {2\pi i m} { c n}  \bigg)  . 
 		\end{align}
 	\end{defn}
 \begin{defn}\label{defn: L(m;c)} Define
 	\begin{align}\label{app: L(c)}
 		L_{\delta} (m;c) =	\frac 1 {\sqrt{c} \varphi (c)} \, \sumd_{   \vchi      (\mathrm{mod} \, c) } \widebar{\vchi} (m)  \tau (\vchi)   L   (   1 / 2 , \vchi   ) .
 	\end{align} 
 \end{defn}
 	The proof of Theorem \ref{thm: explicit} in \S \ref{sec: proof} may be readily modified to yield explicit formulae for $ \EuScript{C}_{\delta} ( m; \phi) $. 
 	
 	\begin{theorem}
 		Let $m \in \BZ_+$ and $\phi (t)$ be as in Theorem  \ref{thm: explicit}. Then 
 		\begin{equation} \label{app: C0}
 		\begin{split}
 			\EuScript{C}_{0} (m ;      \phi) &   =    \EuScript{D} ( m ;     \phi)   -   \EuScript{E}  ( m ;     \phi)    -   \EuScript{E}' ( m ;     \phi)   +   {\EuScript{A}}_{0}^{0} ( m ;     \phi) \\
 			&   +   \breve{\EuScript{A}}_{0}^{1} ( m ;     \phi)   +   \breve{\EuScript{A}}_{0}^{\natural}  ( m ;     \phi)    + \!	\EuScript{K}_{0}^{1}  ( m ;     \phi)   + \!	\EuScript{K}_{0}^{\natural} ( m ;     \phi)  , 
 		\end{split}
 		\end{equation}
 		\begin{equation} \label{app: C1}
 		\begin{split}
 			\EuScript{C}_{1} (m;      \phi)   & =   \EuScript{D} (m;      \phi)    +   {\EuScript{A}}_{1}^{0} (m;      \phi)     \\
 			&  +   \breve{\EuScript{A}}_{1}^{1}  (m;      \phi)    +   \breve{\EuScript{A}}_{1}^{\natural}  (m;      \phi)   +  	\EuScript{K}_{1}^{1}  (m;      \phi)   +  	\EuScript{K}_{1}^{\natural}  (m;      \phi), 
 		\end{split}
 		\end{equation}
 	where 
 	\begin{align}\label{app: D(h)}
 		\EuScript{D} (m; \phi) = \frac 1 {\pi^2 \sqrt{m}} \int_{-\infty}^{\infty}  m^{-it} \phi (t)     \tanh (\pi t) t  \nd t , 
 	\end{align} 
 \begin{align}\label{app: E}
 	\EuScript{E} (m ;  \phi) \! = \! \frac {2 } {\pi} \!  \int_{-\infty}^{\infty} \!   \tau_{it} (m)   \frac  {  \zeta (1/2) \zeta (1/2+2it) } {\, |\zeta (1+2it)|^2 }  \phi (t) \nd t ,  \quad 
	\EuScript{E}'  (m ;   \phi) \! = \!  2\tau_{\frac 1 4} (m)    \frac {   \phi  ( - i /4 )  } { \zeta (3/2)   }  , 
\end{align}
\begin{align}\label{app: A0}
	{\EuScript{A}}_{\delta}^{0} (m; \phi)  =   \frac {2 i^{\delta}  } {\pi^2 \sqrt{m} }  \int_{-\infty}^{\infty} m^{it} {\phi (t)}  \gamma_{\delta} ( 1/2 - 2  i t ) \tanh  (\pi t)     t     \nd t  , 
\end{align}
\begin{align}\label{app: A1}
	\breve{\EuScript{A}}_{\delta}^{1} (m; \phi) =     \frac {2  } {\pi^2 i^{\delta} }  
	\int_{ - \infty}^{ \infty}   {\phi (- t)} \breve{A}_{\delta}  (m; 2it)   \gamma_{1}  (2it) t   \nd t, 
\end{align} 
\begin{align}\label{app: An}
	\breve{\EuScript{A}}_{\delta}^{\natural} (m; \phi) =   \frac { i^{\delta} } {\pi \sqrt{\pi} i}  
	\int_{ - \infty}^{ \infty}    {\phi ( t)} \breve{A}_{\delta}  (m; 2it)   \gamma_{\delta}^{\natural}  (2it) t   \nd t, 
\end{align} 
\begin{align}\label{app: K1}
	\EuScript{K}_{\delta}^{1} (m; \phi) =  \frac {4  } {\pi^2 i^{\delta} }   \sum_{c=1}^{\infty}         {L_{\delta} (m; c)   }    
	\int_{\, i - \infty}^{i+\infty}  {(c/\sqrt{m})^{  2it}}  {\phi (- t)}  \gamma_{1}  (2it) t   \nd t , 
\end{align} 
\begin{align} \label{app: Kn}
	\EuScript{K}_{\delta}^{\natural} ( m; \phi) =   \frac {2 i^{\delta} } {\pi \sqrt{\pi} i}   \sum_{c=1}^{\infty}    {L_{\delta} (m; c )}   
	\int_{\, i - \infty}^{i+\infty}  {(c/\sqrt{m})^{  2it}}  {\phi ( t)} 	\gamma_{\delta}^{\natural} (2it)  t   \nd t . 
\end{align}
 	\end{theorem}
 		
 \end{appendices}


\newcommand{\etalchar}[1]{$^{#1}$}
\def\cprime{$'$}

\end{document}